\documentclass[final]{siamltex}

\usepackage{amsmath,amssymb}
\usepackage{latexsym}
\usepackage{indentfirst}
\usepackage{mathrsfs}
\usepackage{graphicx}
\usepackage{enumerate}
\usepackage[nobysame]{amsrefs}
\usepackage{subfigure}

\usepackage{xcolor}
\usepackage{hyperref}
\usepackage{xr}
\newtheorem{assump}{Assumption}
\newtheorem{assumpmain}{Assumption}

\DeclareMathOperator{\divop}{div}

\newcommand{\ie}{\textit{i.e.}}

\newcommand{\ud}{\,\mathrm{d}}
\newcommand{\RR}{\mathbb{R}}

\newcommand{\ZZ}{\mathbb{Z}}
\newcommand{\TT}{\mathrm{T}}
\newcommand{\LL}{\mathbb{L}}

\newcommand{\Or}{\mathcal{O}}

\newcommand{\bd}[1]{\boldsymbol{#1}}
\newcommand{\wt}[1]{\widetilde{#1}}
\newcommand{\wh}[1]{\widehat{#1}}
\newcommand{\dual}[2]{\langle\,#1,#2\,\rangle}

\IfFileExists{mathabx.sty}%
  {\DeclareFontFamily{U}{mathx}{\hyphenchar\font45}%
   \DeclareFontShape{U}{mathx}{m}{n}{<->mathx10}{}%
   \DeclareSymbolFont{mathx}{U}{mathx}{m}{n}%
   \DeclareFontSubstitution{U}{mathx}{m}{n}%
   \DeclareMathAccent{\widebar}{0}{mathx}{"73}%
}{%
  \PackageWarning{mathabx}{%
    Package mathabx not available, therefore\MessageBreak substituting
    widebar with overline\MessageBreak }%
  \newcommand{\widebar}[1]{\overline{#1}}%
}

\newcommand{\wb}[1]{\widebar{#1}}

\newcommand{\mc}[1]{\mathcal{#1}}

\newcommand{\veps}{\varepsilon}

\newcommand{\ka}{\kappa}

\newcommand{\abs}[1]{\left\lvert#1\right\rvert}
\newcommand{\norm}[1]{\left\lVert#1\right\rVert}

\newcommand{\Lr}[1]{\left(#1\right)}
\newcommand{\inner}[2]{\left\langle#1,#2\right\rangle}

\newcommand{\ffd}[1]{D_{#1}^+}
\newcommand{\bfd}[1]{D_{#1}^-}
\newcommand{\I}{\imath}

\newcommand{\qc}{\mathrm{hy}}
\newcommand{\at}{\mathrm{at}}
\newcommand{\CB}{\mathrm{CB}}

\newcommand{\nn}{\nonumber}

\title{Stability of a
  force-based hybrid method with planar sharp interface
   \thanks{The work of J.L. was partially supported by the Alfred
   P.~Sloan foundation and the National Science Foundation under
  grant DMS-1312659.  He would like to thank helpful discussions
  with Robert V. Kohn. The work of P.B.M. was supported by National
  Natural Science Foundation of China under grants 91230203, and by the funds from Creative Research Groups of China through grant 11321061, and by the support of CAS National Center for Mathematics and Interdisciplinary Sciences.  }
}

\author{Jianfeng Lu\thanks{Departments of Mathematics, Physics, and
    Chemistry, Duke University, Box 90320, Durham, NC, 27708
    USA. Email: jianfeng@math.duke.edu} \and Pingbing
  Ming\thanks{LSEC, Institute of Computational Mathematics and
    Scientific/Engineering Computing, AMSS, Chinese Academy of
    Sciences, No. 55, Zhong-Guan-Cun East Road, Beijing 100190,
    China. Email: mpb@lsec.cc.ac.cn}}

\begin{document}

\maketitle

\begin{abstract}
  We study a force-based hybrid method that couples atomistic model with
  Cauchy-Born elasticity model with sharp transition interface. We
  identify stability conditions that guarantee the convergence of the
  hybrid scheme to the solution of the atomistic model with second
  order accuracy, as the ratio between lattice parameter and the
  characteristic length scale of the deformation tends to
  zero. Convergence is established for hybrid schemes with planar
  sharp interface for system without defects, with general finite
  range atomistic potential and simple lattice structure. The key
  ingredient of the proof is regularity and stability analysis of
  elliptic systems of difference equations. We apply the results to
  atomistic-to-continuum scheme for a 2D triangular lattice with
  planar interface.
\end{abstract}

\begin{keywords} Multiscale method, atomistic-to-continuum, stability
  analysis, force-based coupling
\end{keywords}

\begin{AMS} 65N12; 74S30
\end{AMS}

\pagestyle{myheadings}
\thispagestyle{plain}
\markboth{Lu and Ming}{Force-based hybrid method with planar sharp interface}

\section{Introduction}

Multiscale methods couple together atomistic and continuum models have
received intense investigations in recent years; see, e.g.,~\cites{TadmorOrtizPhillips:1996, AbBrBe98, E:book,TadmorMiller:book, LuskinOrtner}. Generally speaking, there are two main categories of
methods coupling atomistic and continuum models: energy-based methods and
force-based methods. The energy-based methods employ an energy that is
a mixture of atomistic energy and continuum elastic energy. The energy
functional is then minimized subject to suitable boundary conditions
to obtain the deformed state of the system. The force-based methods
work instead at the level of force balance equations. The forces
derived from atomistic and continuum models are coupled together. The
force balance equations supplemented with suitable boundary conditions
are solved to obtain the deformed state of the system.

From a numerical analysis point of view, the key issue for these
multiscale methods is the consistency and stability analysis of the
coupled schemes~\cite{E:book}*{Chapter 7}. In this paper, we study
force-based atomistic-to-continuum hybrid methods in two and three
dimension with sharp transition between the atomistic and continuum
regions. In our previous work~\cite{LuMing:2011}, we developed the
stability analysis in general dimension for a force-based
atomistic-to-continuum method with smooth transition between the two
regions. The main focus of the current paper is to extend the
stability analysis of~\cite{LuMing:2011} to hybrid schemes with sharp
interface between atomistic and continuum models.

Comprehensive reviews for force-based hybrid methods can be found in
\cite{MillerTadmor:2009}*{Section 5 and Section 6}
and~\cite{TadmorMiller:book}*{Section 12.5}. A class of force-based
methods uses a handshake region (transition region) between the
atomistic and continuum regions. A representative of such methods is
the concurrent AtC coupling method (AtC) developed in a series of
papers~\cites{BadiaBochvLehoucqParksFishNuggehallyGunzburger:2007,
  Fish:2007, Badia:2008, Parks:2008}, which blends the continuum stress
and the atomistic forces. The method recently proposed and analyzed by
the authors in~\cite{LuMing:2011} shares certain common traits with
the AtC method. It is proved to be stable and convergent with optimal
convergence rate.  The numerical implementation of the method can be
found in~\cite{YangMingWu:2012}.

As a representative for force-based methods without handshake region,
the FEAt method of Kohlhoff, Gumbsch and
Fischmeister~\cite{KohlhoffGumbschFischmeister:1991} is perhaps one of
the earliest force-based methods. In this method, an elasticity model
is coupled with an atomistic model. The FEAt method does not use a
handshake region: the transition between the atomistic model and the
continuum model is sharp. This kind of coupling is generalized in CADD
method~\cite{ShilkrotMillerCurtin:2002}, which uses the discrete
dislocation model in the continuum region.

One of the main advantages of the force-based methods is that
consistency is achieved with fairly simple construction. Hence the
main focus of analyzing such methods is stability. The stability for
one-dimensional force-based method was already understood in a series
of nice works by Dobson, Luskin, Ortner, and
Shapeev~\cites{DobsonLuskinOrtner:2010a, DobsonLuskinOrtner:2010b,
  DobsonOrtnerShapeev:12}. The generalization to higher dimension is
nontrivial due to the complications of lattice structures and
atomistic interaction potentials. The main idea in our previous
paper~\cite{LuMing:2011} and the current paper is to establish
linearized $H^2$-stability of the hybrid scheme by viewing the scheme
as a nonlinear elliptic finite difference system and applying the
elliptic regularity estimates to such system. The recent
work~\cite{LiLuskinOrtner:12} by Li, Luskin and Ortner proved
linearized $H^1$-stability for methods with smooth coupling under
certain stability conditions. These conditions however were not yet
known how to check explicitly. They also studied how the size of the
transition region affects the stability.

For atomistic-to-continuum hybrid method with sharp interface
studied in this paper, stability might fail at the
interface. To make sure that the hybrid scheme is
convergent, we need to check the stability conditions at the interface
for the coupling schemes. 
We shall identify the interface stability conditions as analog of the
complementing boundary conditions for elliptic PDE system.  From a
physical perspective, these stability conditions amount to check
whether there exists nontrivial surface phonon at the interface of the
two schemes. To some extent, these stability conditions are analogous
to the famous Gustafsson-Kreiss-Sundstr\"om stability
conditions~\cite{GKS:72} for finite difference approximations of mixed
initial/boundary value problems.

The main result in this paper is the linearized $H^2$-stability and
convergence of atomistic-to-continuum hybrid method under the
stability conditions. The essential ingredients are regularity and
stability analysis of finite difference schemes. As a consequence of
our main results, we will show that a force-based
atomistic-to-continuum coupling for a triangular lattice with
next-nearest neighbor harmonic interaction is stable and hence
convergent, when the interface between the atomistic and continuum
regions is planar and is parallel to the $(1,\sqrt{3})/2$ direction of the
lattice. Let us finally remark that while we focus on hybrid schemes
coupling atomistic and nonlinear elasticity models, the ideas and
techniques in the current paper can be extended to other force-based
hybrid methods, e.g., the force-based coupling of peridynamics and
nonlinear elasticity proposed in~\cite{SelesonaBeneddinebPrudhomme:2012}.

\subsection{Atomistic model and Cauchy-Born rule}
We consider classical empirical potentials: For atoms
located at $\{ y_1, \cdots, y_K\}$, the interaction potential energy
between the atoms is given by \( V(y_1,\cdots,y_K), \) which often
takes the form
\[
V(y_1,\cdots,y_K)=\sum_{i,j}V_2(y_i/\veps, y_j/\veps)
+\sum_{i,j,k}V_3(y_i/\veps,y_j/\veps,y_k/\veps)+\cdots.
\]
As in~\cite{EMing:2007} and our previous work~\cite{LuMing:2011}, we
make the following assumptions on the potential function $V$: $V$ is
invariant with respect to translations and rigid body motion; $V$ is
smooth in a neighborhood of the equilibrium state; and $V$ has finite
range.  For simplicity of notation and clarity of presentation, our
presentation will be limited to potentials that contain only two-body
and three-body potentials, and we will only make explicit the
three-body terms in the expressions of the potential.  
As the
potential function $V$ is a function of atom distances and angles by
invariance with respect to rigid body motion, we may write
\begin{equation*}
  V_3(y_i,y_j,y_k) =V_3\Lr{\abs{y_i-y_j}^2,
    \abs{y_i-y_k}^2,\inner{y_i-y_j}{y_i-y_k}},
\end{equation*}
where $\inner{\cdot}{\cdot}$ denotes the inner product over $\RR^d$.

We denote $\Omega_{\veps}$ the collections of atom positions in
equilibrium, with $x\in\Omega_{\veps}$ denotes the equilibrium
position of individual atom. Positions of the atoms under deformation
will be viewed as a function defined over $\Omega_{\veps}$, which is
denoted as $y(x) = x + u(x)$. Hence, $u: \Omega_{\veps} \to \RR^d$ is
the displacement of the atoms. We will use the same notations for
lattice functions and theirs norms as in our previous work
\cite{LuMing:2011} (also recalled in the Supplementary Materials
Section~\ref{sec:lattice} for readers' convenience).  Define the space
of the displacement field as
\[
X_{\veps} = \Bigl\{ u: \Omega_{\veps} \to \RR^d \,\Big\vert\,
\sum_{x\in\Omega_{\veps}} u(x) = 0\Bigr\}.
\]
The atomistic problem is formulated as follows. Given force field
$f_{\veps}: \Omega_{\veps} \to \RR^d$, find $u \in X_{\veps}$ such
that
\begin{equation}\label {atom:min}
  u = \arg\min_{u \in X_{\veps}} I_{\at}(u),
\end{equation}
where
\begin{equation*}
  I_{\at}(u) = \dfrac{1}{3!}  \veps^d \sum_{x\in\Omega_{\veps}}\sum_{(s_1,s_2)
    \in S} V_{(s_1, s_2)}[x+u] - \veps^d \sum_{x \in
    \Omega_{\veps}} f_{\veps}(x) u(x),
\end{equation*}
and
\begin{equation*}
  V_{(s_1, s_2)}[y] =
  V\Lr{\abs{\ffd{\veps,s_1}y(x)}^2,
    \abs{\ffd{\veps,s_2}y(x)}^2,\inner{\ffd{\veps,s_1}y(x)}{\ffd{\veps,s_2}y(x)}}.
\end{equation*}
Here $S$ is the set of all possible $(s_1, s_2)$ within the range of
the potential. By our assumptions, $S$ is a finite set. 
In $I_{\at}$, $\veps^d$ is a normalization factor, so that $I_{\at}$
is actually the energy of the system per atom.

The Euler-Lagrange equation for the atomistic problem is
\begin{equation}\label{atom:eq}
  \mc{F}_{\at}[u](x) = f_{\veps}(x)\qquad x\in\Omega_{\veps},
\end{equation}
where
\begin{align*}
    \mc{F}_{\at}[u](x) &= \sum_{(s_1,s_2)\in S}
    \Bigl( \bfd{\veps,s_1}\Lr{2\partial_1
      V_{(s_1, s_2)}[y](x)\ffd{\veps,s_1}y(x)+\partial_3V_{(s_1, s_2)}[y](x)\ffd{\veps,s_2}y(x)}\\
    &\phantom{\sum_{(s_1,s_2)}}\qquad+\bfd{\veps,s_2}\Lr{2\partial_2 V_{(s_1,
        s_2)}[y](x)\ffd{\veps,s_2}y(x)+\partial_3V_{(s_1,
        s_2)}[y](x)\ffd{\veps,s_1}y(x)}\Bigr),
\end{align*}
where for $i=1,2,3$, we denote
\[
\partial_i V_{(s_1, s_2)}[y](x)=\partial_i
V\Lr{\abs{\ffd{\veps,s_1}y(x)}^2,\abs{\ffd{\veps,s_2}y(x)}^2,
  \inner{\ffd{\veps,s_1}y(x)}{\ffd{\veps,s_2}y(x)}}
\]
the partial derivative with respect to the $i$-th argument of $V$.

To guarantee the solvability of \eqref {atom:eq}, we assume that
$f_{\veps}$ takes the following form:
\[
f_{\veps}(x)\equiv\veps^{-d}\int_{x+\veps\Gamma} f(z) \ud z,\quad
x\in\Omega_\veps,
\]
where $f(x)$ is a function defined on $\Omega$ with zero mean. This
makes sure that $f_{\veps}(x)$ satisfies
\[
\sum_{x\in\Omega_\veps}f_{\veps}(x)=\veps^{-d}\int_{\Omega} f(x) \ud
x=0.
\]

To introduce the Cauchy-Born elasticity problem \cites{BornHuang:1954,
  Ericksen:1984, Ericksen:2008}, we fix more notations. For any
positive integer $k$, we denote by $W^{k,p}(\Omega;\RR^d)$ the Sobolev
space of mappings $u{:}\;\Omega\to\RR^d$ such that $\|u\|_{W^{k,p}}$
is finite, and by $W_{\#}^{k,p}(\Omega;\RR^d)$ the Sobolev space of
periodic functions whose distributional derivatives of order less than
$k$ are in $L^p(\Omega)$. For any $p>d$ and $m\ge 0$, we define $X$
as
\[
X=\Bigl\{u: \Omega \to \RR^d\,\Big\vert\,u\in W^{m+2,p}(\Omega;\RR^d)
\cap\,W_{\#}^{1,p}(\Omega;\RR^d), \, \int_{\Omega}u=0\Bigr\}.
\]
The Cauchy-Born elasticity problem is formulated as follows. Find
$u\in X$ such that
\begin{equation}\label {cb:min}
  u = \arg \min_{u\in X}I(u),
\end{equation}
where the total energy functional $I$ is given by
\[
I(u)=\int_{\Omega} \bigl(W_{\CB}(\nabla u(x))- f(x)u(x)\bigr)\ud x.
\]
Here the Cauchy-Born stored energy density $W_{\CB}$ is given by
\[
W_{\CB}(A)=\dfrac{1}{3!}\sum_{(s_1,s_2)\in S} W_{(s_1, s_2)}(A),
\]
where for $A \in \RR^{d\times d}$,
\[
W_{(s_1, s_2)}(A) = V\Lr{\abs{s_1+s_1
    A}^2,\abs{s_2+s_2A}^2,\inner{s_1+s_1A}{s_2+s_2A}}.
\]
The range $S$ is the same as that in the atomistic potential function.

The Euler-Lagrange equation for the Cauchy-Born elasticity
model is
\begin{equation}\label {cb:eq}
  \mc{F}_{\CB}[u](x)=f(x),
\end{equation}
where $\mc{F}_{\CB}[u]=-\divop\bigl(D_AW_{\CB}(\nabla u)\bigr)$ with
$D_A W_{\CB}(A)$ denoting the derivative of $W_{\CB}(A)$ with respect
to $A$.  Since we are primarily interested in the coupling between the
atomistic and continuum models, we will take the finite element mesh
$\mathcal{T}_{\veps}$ as a triangulation of $\Omega_{\veps}$ with each
atom site as an element vertex.
The triangulation is translational invariant. The approximation space
$\wt{X}_{\veps}$ is defined as
\[
\wt{X}_{\veps}=\bigl\{ u \in W^{1,p}_{\#}(\Omega;\RR^d) \mid
u|_T\in P_1(T), \ \forall\, T\in\mc{T}_{\veps}\bigr\},
\]
where $P_1(T)$ is the space of linear functions on the element $T$.
We denote by $\mc{F}_{\veps}$ the force from finite element approximation
of Cauchy-Born elasticity problem~\eqref{cb:min}.
\subsection{Formulation of force-based hybrid method with sharp
  interface}\label{sec:formulation}
To formulate the force-based hybrid method, we take a
continuum region
\begin{equation*}
  \Omega_c = \Bigl\{ \sum_{j=1}^d x_ja_j \,\Big\vert\,0 \leq x_1 < 1/2,\, 0
  \leq x_j < 1,\, j = 2, \cdots, d \Bigr\},
\end{equation*}
and denote $\varrho$ the characteristic function associated with
$\Omega_c$: $\varrho(x) = 1$ if $ x \in \Omega_c$. $\Omega_a = \Omega
\backslash \Omega_c = \{ x \mid \varrho(x) = 0\}$ is the atomistic
region. The continuum region and atomistic region are separated by
two hyperplanes $\{x_1 = 0\}$ and $\{x_1 = 1/2\}$ as a result of
periodic boundary condition. The simple geometry here is chosen for
simplicity of presentation. Using localization techniques as in~\cite{LaxNirenberg:1966},
we may generalize the analysis to any $\Omega_c$ with smooth boundary.

We consider a force field defined by
\begin{equation}\label{eq:sharpforce}
  \mc{F}_{\qc}[u](x)\equiv \bigl(1 - \varrho(x)\bigr) \mc{F}_{\at}[u](x)
  + \varrho(x) \mc{F}_{\veps}[u](x), \qquad
  x\in\Omega_{\veps}.
\end{equation}
Due to the choice of $\varrho$, in the atomistic region $\Omega_{a}$,
the force acting on the atom is just that of atomistic model, while in
the continuum region $\Omega_c$, the force is calculated from finite
element approximation of the Cauchy-Born elasticity. Since
$\varrho$ is taken to be the characteristic function, we consider here
a hybrid method with sharp interface, i.e., there is no transition or
buffer region between the atomistic and continuum regions.

Given a loading $f_{\veps}$, we find $u \in X_{\veps}$ such that
\begin{equation}\label{eq:sharp}
  (\Pi_{\veps} \mc{F}_{\qc}[u])(x) = f_{\veps}(x)\qquad x \in \Omega_{\veps},
\end{equation}
where for a lattice function $g$, $\Pi_{\veps}$ projects $g$ to a
function with zero mean.
\[
  (\Pi_{\veps} g)(x) {:}= g(x) - \veps^d \sum_{x' \in \Omega_{\veps}} g(x').
\]

As in~\cite{LuMing:2011}, the convergence of the hybrid
scheme is tightly connected with its linear stability.
Thus, it is natural to study the linearized operator of
$\mc{F}_{\qc}$.  Denote $\mc{H}_{\qc}[u]$ the linearization of
$\mc{F}_{\qc}$ at state $u$: $\mc{H}_{\qc}[u] =\delta
  \mc{F}_{\qc}/\delta u$, so that $\mc{H}_{\qc}[u]$ is a linear
operator acting on a lattice functions $w$, which is given by
\[
\mc{H}_{\qc}[u] w = \lim_{t\to 0} \dfrac{\partial \mc{F}_{\qc}}{\partial t}[u + t w].
\]
We will rewrite the operator $\mc{H}_{\qc}$ in the form of a
difference operator as
\[
  \mc{H}_{\qc}[u] = \sum_{\mu \in \mc{A}} h_{\qc}[u](x, \mu) T^{\mu},
\]
where the coefficient $h_{\qc}[u](x,\mu)$ is a $d$ by $d$ matrix (probably
asymmetric) for each $x\in\Omega_{\veps}$ and $\mu \in \mc{A}$, which is given by
\begin{equation}\label{eq:defhqc}
  (h_{\qc}[u])_{\alpha\beta}(x, \mu) = \frac{\partial
    ( \mc{F}_{\qc}[u] )_{\alpha}(x)}
  {\partial (T^{\mu} u)_{\beta}(x)},\quad\alpha, \beta = 1, \cdots, d.
\end{equation}
Here $\mc{A}$ is the
stencil of the difference operator, which is finite by assumptions on the atomistic potential. By the definition of $\mc{F}_{\qc}$, we have
\begin{equation}\label{eq:linearh}
  h_{\qc}[u](x, \mu) = \bigl(1 - \varrho(x)\bigr) h_{\at}[u](x, \mu)
  + \varrho(x) h_{\veps}[u](x, \mu),
\end{equation}
where $h_{\at}[u]$ and $h_{\veps}[u]$ are given by similar equations
as \eqref{eq:defhqc} with $\mc{F}_{\qc}$ replaced by $\mc{F}_{\at}$
and $\mc{F}_{\veps}$, respectively.

Define $h_{\qc}[u](x, \xi)$ as the symbol of the pseudo-difference
operator $\mc{H}_{\qc}[u]$, which is given by
\[
h_{\qc}[u](x, \xi) = \sum_{\mu\in\mc{A}} h_{\qc}[u](x, \mu)
  \exp\Bigl(\I \veps \sum_j \mu_j a_j \cdot \xi\Bigr) \qquad \text{for } \xi \in
  \LL^{\ast}_{\veps},
\]
and similarly for $h_{\veps}[u]$ and $h_{\at}[u]$.  By
definition, we have for any $x \in \Omega_{\veps}$,
\[
  (\mc{H}_{\qc}[u] e_k e^{\I x\cdot \xi})_j(x) =
  (h_{\qc}[u])_{jk}(x, \xi) e^{\I x\cdot \xi},
\]
for $j,k =1,\dots,d$ and similarly for $h_{\veps}[u]$ and
$h_{\at}[u]$.  It is also clear that~\eqref{eq:linearh} implies
\begin{equation}\label{eq:linearwth}
  h_{\qc}[u](x, \xi) = (1 - \varrho(x)) h_{\at}[u](x, \xi)
  + \varrho(x) h_{\veps}[u](x, \xi).
\end{equation}

When $\mc{F}_{\qc}$ is linearized around the equilibrium state $u = 0$,
we will simplify the notation as $\mc{H}_{\qc} = \mc{H}_{\qc}[0]$,
$h_{\qc} = h_{\qc}[0]$, and similarly for those defined for atomistic
model and finite element discretization of the Cauchy-Born
elasticity. By the translation invariance of the total energy $I_\at$
at the state $u = 0$, we observe that the coefficients of the symbols
$h_{\at}(x, \mu)$ and $h_{\veps}(x, \mu)$ are independent of the
position $x$, i.e.,
\[
h_{\at}(x, \mu) = h_{\at}(\mu), \quad h_{\veps}(x, \mu) =
h_{\veps}(\mu).
\]

We also denote $\mc{H}_{\CB}$ as the linearization of $\mc{F}_{\CB}$
at the equilibrium state $u = 0$, and define $h_{\CB}(x, \xi)$ as its
symbol. Due to the periodic boundary condition, the argument $\xi$ in the symbol $h_{\CB}(x, \xi)$ only takes value in $\LL^{\ast}$. Again, due to
the translation invariance of the total energy, the symbol
$h_{\CB}$ is also independent of $x$.

An elementary calculation shows that the matrices $h_{\at}(\xi)$,
$h_{\veps}(\xi)$ and hence $h_{\qc}(x, \xi)$ are Hermitian for any
$\veps > 0$. As in~\cite{LuMing:2011}, we make the following stability
assumption on the atomistic potential:
\begin{assumpmain}\label{assump:stabatom}
  The matrix $h_{\at}(\xi)$ is positive definite and there exists
  a positive constant $a_{\at}$ such that for any $\veps>0$ and any $\xi\in
  \LL_{\veps}^{\ast}$,
  \begin{equation*}
    \det {h}_{\at}(\xi) \geq a_{\at} \Lambda_{0, \veps}^{2d}(\xi).
  \end{equation*}
\end{assumpmain}
%
%
\subsection{Stability conditions at the interface}
The main focus of the current paper is to establish convergence for
the hybrid method with sharp interface as $\veps \to 0$. In
\cite{LuMing:2011}, convergence was proved for any short-range
interaction potentials when $\varrho$ is a smooth function, or in
other words, when the transition region between atomistic and
continuum regions is of $\Or(\veps^{-1})$. In this paper, as the
transition is sharp, we require additional stability conditions to
guarantee convergence.

To understand better where the additional stability assumptions come
from, let us reformulate the hybrid scheme as a system of difference
equations with boundary conditions by folding with respect to the
interface (similar folding trick was used in \cite{Ciment:1971}). This
is one of the keys to establish the stability for the hybrid
scheme. Let us consider domain
\begin{equation*}
  \Omega_{\mathrm{strip}}
  = \Bigl\{ x a_1 + \sum_j y_j a_{j+1} \mid x \in \RR,\, y \in [0,
  1)^{d-1} \Bigr\}
\end{equation*}
with periodic boundary condition in $y$
variable. $\Omega_{\mathrm{strip}}$ is discretized by grid points
$(x_{\nu}, y_{\mu})$ with $x_{\nu} = \veps \nu a_1$ with
$\nu \in \ZZ$ and
$y_{\mu} = \veps \sum_j \mu_j a_{j+1}$ with $\mu \in \ZZ^{d-1}$. We
consider the following hybrid system on $\Omega_{\mathrm{strip}}$:
\begin{align}\label{eq:Lcorig}
  & \sum_{j=1}^d (\mc{H}_{\veps, ij} u_j)(x_{\nu}, y_{\mu}) =
  f_{i}(x_{\nu}, y_{\mu}), && \nu < 0, \ i = 1, \cdots, d; \\
  \label{eq:Laorig}
  & \sum_{j=1}^d (\mc{H}_{\at, ij} u_j)(x_{\nu}, y_{\mu}) =
  f_{i}(x_{\nu}, y_{\mu}), && \nu \geq 0, \ i = 1, \cdots, d.
\end{align}
As we will see later, the stability analysis of the coupled system is the key to understand the stability of the hybrid scheme with sharp interface.

Let $(\underline{\nu}^c_{ij}, \overline{\nu}^c_{ij})$ and
$(\underline{\nu}^a_{ij}, \overline{\nu}^a_{ij})$ be the extent of the
stencils of $\mc{H}_{\veps, ij}$ and $\mc{H}_{\at, ij}$ in $x$
direction respectively.  To simplify the presentation, we will assume
that the extents are the same for all $i, j$'s, which is usually the
case for applications in atomistic-continuum hybrid schemes. We denote
them as $(\underline{\nu}^c, \overline{\nu}^c)$ and
$(\underline{\nu}^a,\overline{\nu}^a)$. The construction can be
extended to more general cases, which will be omitted for simplicity.
Without loss of generality, we also assume that $\underline{\nu}^a
\leq \underline{\nu}^c \leq 0 \leq \overline{\nu}^c \leq
\overline{\nu}^a$.

To do the folding, we rename the variables as
\begin{align*}
  & U_i(x_{\nu}, y_{\mu}) = u_i(x_{\overline{\nu}^c - \nu - 1},
  y_{\mu}), & \nu \geq 0, \mu \in \RR^{d-1}; \\
  & U_{d+i}(x_{\nu}, y_{\mu}) = u_i(x_{\underline{\nu}^a+\nu},
  y_{\mu}), & \nu \geq 0, \mu \in \RR^{d-1}.
\end{align*}
This leads to $d (\overline{\nu}^c - \underline{\nu}^a)$ compatibility
conditions
\begin{equation}\label{eq:compat}
  U_i(x_{\nu}, y_{\mu}) = U_{d+i}(x_{\overline{\nu}^c - \underline{\nu}^a-\nu-1}, y_{\mu}),
  \quad
  0 \leq \nu \leq \overline{\nu}^c - \underline{\nu}^a - 1,\, i = 1, \cdots, d.
\end{equation}
We will rewrite equations \eqref{eq:Lcorig}-\eqref{eq:Laorig} in terms
of $U$'s. Define
\begin{equation*}
  L_{\veps, ij} =
  \begin{cases}
    \displaystyle \sum_{0 \leq \nu \leq \overline{\nu}^c -
      \underline{\nu}^c, \mu} h_{\veps, ij}(\overline{\nu}^c-\nu,\,
    \mu) T_{\veps, x}^{\nu} T_{\veps, y}^{\mu}, & i, j = 1, \cdots, d; \\
    \displaystyle \sum_{\underline{\nu}^a \leq \nu \leq
      \overline{\nu}^a,\, \mu} h_{\at, (i-d)(j-d)}(\nu, \mu) T_{\veps,
      x}^{\nu} T_{\veps, y}^{\mu}, & i, j = d+1, \cdots, 2d; \\
    0, & \text{otherwise.}
  \end{cases}
\end{equation*}
By construction, the extents of the operators $L_{\veps}$ are given by
\begin{equation*}
  (\underline{\nu}_{ij}, \overline{\nu}_{ij}) =
  \begin{cases}
    (0, \overline{\nu}^c -
    \underline{\nu}^c), & i,j = 1, \cdots, d; \\
    (0, \overline{\nu}^a - \underline{\nu}^a), &
    i,j = d+1, \cdots, 2d; \\
    (0, 0), & \text{otherwise.}
  \end{cases}
\end{equation*}
Let
\begin{equation*}
  F_i(x_{\nu}, y_{\mu}) =
  \begin{cases}
    f_i(x_{-\nu-1}, y_{\mu}), & i = 1, \cdots, d; \\
    f_{i-d}(x_{\nu}, y_{\mu}), & i = d+1, \cdots, 2d,
  \end{cases}
\end{equation*}
we then have
\begin{equation}\label{eq:L}
  \sum_{j=1}^{2d} (L_{\veps, ij} U_j)(x_{\nu}, y_{\mu}) = F_i(x_{\nu}, y_{\mu}),
  \qquad \nu \geq 0, \mu \in \ZZ^{d-1}.
\end{equation}
We further define for $k = 1, \cdots, d(\overline{\nu}^c -
\underline{\nu}^a)$ and $j = 1, \cdots, d$,
\begin{equation*}
  B_{\veps, kj} = (D_{\veps, e_1}^+)^i I \delta_{jl}, \qquad B_{\veps, k(j+d)} =
  - (D_{\veps, -e_1}^+)^i T_{\veps, x}^{\overline{\nu}^c - \underline{\nu}^a - 1} \delta_{jl},
\end{equation*}
where $i = \lfloor (k-1) / d \rfloor$ and $l = [(k-1) \mod d] + 1$.
The compatibility conditions~\eqref{eq:compat} are then equivalent to
the boundary conditions
\begin{equation}\label{eq:B}
  \sum_{j=1}^{2d} (B_{\veps, kj} U_j)(x_0, y_{\mu}) = 0, \qquad \mu \in \ZZ^{d-1},
  k = 1, \cdots, d(\overline{\nu}^c - \underline{\nu}^a).
\end{equation}
Therefore, we have reformulated \eqref{eq:Lcorig}-\eqref{eq:Laorig}
into a difference system \eqref{eq:L} with boundary conditions
\eqref{eq:B}.

We now state the additional stability conditions, which can be understood as conditions to make sure that the solutions to the finite difference system is regular up to the boundary.
Taking $\alpha^+ = 0$, $\beta_j^+ = \overline{\nu}^c -
\underline{\nu}^c$, $j = 1, \cdots, d$, and $\beta_j^+ =
\overline{\nu}^a - \underline{\nu}^a$, $j = d + 1, \cdots, 2d$, we
have then $0 \leq \underline{\nu}_{ij} \leq \overline{\nu}_{ij} \leq
\alpha_i^+ + \beta_j^+$. The number of boundary conditions we have
imposed on~\eqref{eq:B} is $q = d (\overline{\nu}^c - \underline{\nu}^a)$.
\begin{assumpmain}\label{assump:boundarycond}
  The number of boundary conditions $d(\overline{\nu}^c -
  \underline{\nu}^a)$ is equal to the total number of roots of
  \begin{equation*}
    \begin{aligned}
      & R^c(z, \eta) = \det \left[ h_{\veps, ij}(-i \veps^{-1} \log z,
        \veps^{-1} \eta) \right]  = 0; \\
      & R^a(z, \eta) = \det \left[ h_{\at, ij}(-i \veps^{-1} \log z,
        \veps^{-1} \eta) \right] = 0.
     \end{aligned}
  \end{equation*}
  satisfying
  $0 < \abs{z(\eta)} < 1$ for $\abs{\eta} \neq 0$.
\end{assumpmain}


The stability of the hybrid scheme relies on the stability at
interface, which is characterized by the complementing boundary
condition.
\begin{assumpmain}\label{assump:complement}
  The hybrid system \eqref{eq:Lcorig}-\eqref{eq:Laorig}, or
  equivalently \eqref{eq:L}-\eqref{eq:B} satisfies the
  \emph{Complementing Boundary Condition}, defined below in
  Section~\ref{sec:complement}.
\end{assumpmain}

The complementing boundary conditions are natural analog of
corresponding complementing boundary conditions for continuous
elliptic system, as in \cite{AgmonDouglisNirenberg:1959}, which will
be explained in details in Section~\ref{sec:complement}.  This
stability assumption needs to be checked for particular atomic
interaction potentials. Examples can be found in
Section~\ref{sec:example}.
\subsection{Main result}
The main result of this paper is
\begin{theorem}[Convergence]\label{thm:main}
  Under Assumptions~\ref{assump:stabatom}, \ref{assump:boundarycond},
  and \ref{assump:complement}, there exist positive constants $\delta$
  and $M$, so that for any $p > d$ and $ f \in W^{15, p}(\Omega) \cap
  W^{1, p}_{\#}(\Omega) $ with $\norm{f}_{W^{15, p}} \leq \delta$, we
  have $\norm{y_{\qc} - y_{\at}}_{\veps,2} \leq M \veps^2$.
\end{theorem}

\smallskip

\textit{Remark.}  The Assumption~\ref{assump:stabatom} is a natural
stability condition for the atomistic lattice system. Compared with
the convergence result in \cite{LuMing:2011}, the additional
assumptions \ref{assump:boundarycond} and \ref{assump:complement} come
from the coupling between the atomistic potential and the finite
element discretization at the interface. 



The proof of Theorem~\ref{thm:main} follows a similar strategy as
in~\cite{LuMing:2011}. Actually, once we obtain the stability estimate
Theorem~\ref{thm:stability}, the proof of Theorem~\ref{thm:main} is
essentially the same as that of \cite{LuMing:2011}*{Theorem 1.1}. The
consistency analysis of the scheme follows from that
of~\cite{LuMing:2011}*{Section 2} with some immediate
adaptations. Observe in particular that the proof of consistency does
not depend on the smoothness of $\varrho$. Hence, we will focus on the
linear stability analysis, and omit the consistency
part.

The remaining of the paper is organized as follows. In
Section~\ref{sec:swb}, we recall the regularity theory for elliptic
difference system with boundary conditions and use it to establish the
regularity estimates for the hybrid scheme.  The stability estimate then follows from the
regularity estimate combined with consistency, which is presented in
Section~\ref{sec:stability}. In Section~\ref{sec:example}, we apply
the general theory to two examples of force-based atomistic-to-continuum
hybrid method for $2D$ triangular lattice with next-nearest neighbor
harmonic interaction and truncated Lennard-Jones potential. We make some conclusive remarks in
Section~\ref{sec:conclusion}.
\section{Regularity estimate of hybrid schemes}\label{sec:swb}

To analyze the general hybrid scheme, we will take the viewpoint of
our previous work~\cite{LuMing:2011} and regard the scheme as a
nonlinear finite difference scheme. One of the main ingredients we
will use in this paper is the regularity estimates up to the boundary
for elliptic difference systems established
in~\cite{StrikwerdaWadeBube:1990}. We also note that the regularity
estimates of elliptic difference equations and systems have been
investigated by several works~\cites{LaxNirenberg:1966,
  ThomeeWestergren:1968, Hackbusch:1981, Hackbusch:1983,
  BubeStrikwerda:1983, StrikwerdaWadeBube:1990, Martin:94}.  For
reader's convenience, we recall here briefly  the setup and results with some
adaptation to the current work.

\subsection{Difference operators}
We write a difference operator $L_{\veps}$ in the form of
\begin{equation}\label{eq:transop}
  (L_{\veps} u)(x) = \sum_{\mu \in \mc{A}} l_{\veps}(x, \mu) (T^{\mu}_{\veps} u)(x),
\end{equation}
where the stencil $\mc{A} \subset \ZZ^d$ is finite. We define the
symbol of $L_{\veps}$ as\footnote{To be consistent
  with~\cite{LuMing:2011}, our scaling of the reciprocal space in
  terms of $\veps$ is slightly different from that
  of~\cite{StrikwerdaWadeBube:1990}.}
\begin{equation}\label{eq:defsymbol}
  l_{\veps}(x, \zeta) = \sum_{\mu \in \mc{A}} l_{\veps}(x, \mu)
  e^{\I 2\pi \veps \mu \cdot \zeta}
\end{equation}
for any $\zeta \in \ZZ^d$. By definition, we have
\begin{equation*}
  \Bigl(L_{\veps} e^{\I x \cdot (\zeta_j b_j)}\Bigr)(x) = l_{\veps}(x, \zeta) e^{\I x\cdot(\zeta_j b_j)}.
\end{equation*}
We choose $\mc{A}$ as the minimal stencil, that is, for $\mu \in \mc{A}$,
the symbol $l_{\veps}(x, \mu)$ is not identically zero. For this choice of
$\mc{A}$, we let
\begin{equation*}
  \underline{\mu}_i = \min_{\mu \in \mc{A}} \mu_i,
  \quad \overline{\mu}_i = \max_{\mu \in \mc{A}} \mu_i,
  \qquad i = 1, \cdots, d,
\end{equation*}
and call $(\underline{\mu}_i, \overline{\mu}_i)$ the \emph{extent} of
the difference operator in $x_i$ direction.
\subsection{Elliptic difference system}
Consider a half-space
\begin{equation*}
  \Omega_{\mathrm{half}}
  = \Bigl\{ x a_1 + \sum_j y_j a_{j+1} \mid x \geq 0, y \in [0,1)^{d-1} \Bigr\}
\end{equation*}
with periodic boundary condition in $y$. $\Omega_{\mathrm{half}}$ is
discretized by grid points $(x_{\nu}, y_{\mu})$ with mesh size $\veps
= 1/ (2N)$. Here $x_{\nu} = \veps \nu a_1$ with $\nu \geq 0$ and $y_{\mu} =
\veps \sum_j \mu_j a_{j+1}$ with $\mu \in \ZZ^{d-1}$.
Consider a linear system of difference equations with boundary conditions
\begin{align}\label{eq:Lsys}
  & \sum_{j=1}^n (L_{\veps, ij} u_j)(x_{\nu}, y_{\mu}) = f_i(x_{\nu},
  y_{\mu}),\qquad i = 1, \cdots, n \\
  & \label{eq:Bsys} \sum_{j=1}^n (B_{\veps, kj} u_j)(x_0, y_{\mu}) =
  0,\qquad k = 1, \cdots, q.
\end{align}
We denote $l_{\veps, ij}$ the symbol of $L_{\veps, ij}$ and $b_{\veps,
  kj}$ the symbol of $B_{\veps, kj}$ viewed as difference operators on
the whole space, defined as in \eqref{eq:defsymbol}. For our purpose,
it suffices to consider the case that the coefficients of difference
operators are independent of $x$, $y$, and $\veps$, \ie,
\begin{equation*}
  l_{\veps, ij}(x, y, \nu, \mu) = l_{ij}(\nu, \mu),
  \qquad b_{\veps, kj}(y, \nu, \mu) = b_{kj}(\nu, \mu).
\end{equation*}
We will henceforth make this simplification. Note that the symbols
$l_{\veps, ij}$ and $b_{\veps, kj}$ however still depend on $\veps$.
\begin{definition}[Regular elliptic difference system]
  We call~\eqref{eq:Lsys} a regular elliptic of order $(\sigma, \tau)$
  for $\sigma, \tau \in \ZZ^n$ if the following conditions are satisfied.
  \begin{enumerate}[i)]
  \item For each $i, j = 1, \cdots, n$ and $\veps$ sufficiently small,
    $\abs{l_{\veps, ij}(\zeta)} \lesssim
      \Lambda_{\veps}(\zeta)^{\sigma_i + \tau_j}$;
  \item There exist positive constants $\zeta_0$ and $\veps_0$ such
    that $\abs{\det l_{\veps}(\zeta)} \gtrsim
    \Lambda_{\veps}(\zeta)^{2p}$ for all $0 < \veps \leq \veps_0$ and
    $\abs{\zeta} \geq \zeta_0$, where $2p = \sum_i (\sigma_i +
    \tau_i)$.
  \end{enumerate}
\end{definition}
As a convention, we will choose $\sigma$ and $\tau$ so that $\max
\sigma_i = 0$. We also take $\rho \in \ZZ^q$ so that $\abs{b_{\veps,
    kj}(\zeta)} \lesssim \Lambda_{\veps}(\zeta)^{\rho_k + \tau_j}$,
with $\rho_k$ is the smallest possible integer such that the estimate
hold. $(\rho, \tau)$ gives the order of $B_{\veps}$ viewing as
difference operators on the whole space. We will assume that the
operators $L_{\veps, ij}$ and $B_{\veps, kj}$ only contain differences
of order $\sigma_i + \tau_j$ and $\rho_k + \tau_j$, respectively.
\begin{assump}\label{assump:boundary}
  Let the difference operators $L_{\veps, ij}$ have extent
  $(\underline{\nu}_{ij}, \overline{\nu}_{ij})$ in $x-$direction, we
  assume that there are $\alpha^+, \beta^+$ in $\ZZ^d$ such that
  \begin{equation*}
    0 \leq \underline{\nu}_{ij} \leq \overline{\nu}_{ij}
    \leq \alpha_i^+ + \beta_j^+, \qquad
    i, j = 1, \cdots, n,
  \end{equation*}
  and such that the number of roots $z(\eta)$, counting multiplicity,
  of the equation $R(z, \eta) = 0$ is $\sum_i (\alpha_i^+ +
  \beta_i^+)$. Here
  \[
  R(z, \eta) = \det \biggl[ \sum_{\nu \in \ZZ}
        \sum_{\mu\in\ZZ^{d-1}} l_{ij}(\nu, \mu) z^{\nu} e^{\imath\mu\cdot\eta}\biggr]
        = \det \bigl[ l_{\veps, ij}(-\imath \veps^{-1} \log z, \veps^{-1}\eta) \bigr].
  \]
  Furthermore, the number of boundary conditions $q$ is equal to the
  number of roots of $R(z, \eta) = 0$ that satisfy
  $0 < \abs{z(\eta)} < 1$ for $\abs{\eta} \neq 0$.
\end{assump}

When $q$ is larger than $p$, an additional
assumption is needed. We assume that the boundary conditions are ordered so that
$\rho_k$ are in increasing order. We define
\begin{equation*}
  \wb{\rho} = \max_{1\leq k \leq p}(\rho_k + 1, 0),
  \quad \text{and} \quad  \rho^{\ast} =
 \begin{cases} \displaystyle
   \min_{k>p} \rho_k + 1, & \text{if } q > p\,; \\
   \infty, & \text{if }q = p\,.
 \end{cases}
\end{equation*}
\begin{assump}\label{assump:boundweight}
  The boundary operators satisfy $\rho_k \geq \wb{\rho}$ for
  $p<k\leq q$.
\end{assump}

In the limit $\veps \to 0$, we obtain the system of differential equations associated with the system of difference equations.
\begin{align*}
  & \sum_{j=1}^n (L_{ij}(\partial_x, \partial_y) u_j)(x, y) =
  f_i(x, y) &&  i = 1, \cdots, n,\\
  & \sum_{j=1}^n (B_{kj}(\partial_x, \partial_y) u_j)(0, y) = 0 && k =
  1, \cdots, p.
\end{align*}
The symbol of $L_{ij}$ is given by $l_{ij}(\zeta) = \lim_{\veps \to 0} l_{\veps, ij}(\zeta)$.
\begin{assump}[Supplementary condition]
  \label{assump:supple}
  The determinant $\det [l_{ij}(\zeta)]$ is of even degree $2p$. For every pair of
  vectors $\zeta, \zeta' \in \RR^d$, the polynomial $\det
  [l_{ij}(\zeta + \tau \zeta')]$ in the complex variable $\tau$ has
  exactly $p$ roots with positive imaginary part.
\end{assump}

\subsection{Complementing boundary condition}\label{sec:complement}

Due to periodicity in $y$, we may reduce the system to a one
dimensional operator by Fourier transform. Let $G$ be the tangential
grid given by $G = \bigl\{ y_{\mu} \mid 0 \leq \mu_j < 2N, j = 1,
\cdots, d -1 \bigr\}$.  The discrete Fourier transform for lattice
functions $u$ with respect to $y$ is then
\begin{equation*}
  \wh{u}(x_{\nu}, \eta) = \Bigl(\frac{\veps}{2\pi}\Bigr)^{d-1}
  \sum_{y_{\mu} \in G} e^{-\I 2\pi \veps \eta \cdot \mu} u(x_{\nu}, y_{\mu}),
\end{equation*}
for $\eta \in \ZZ^{d-1}$ with $\abs{\eta_j} \leq N$ for $j = 1,
\cdots, d-1$.  
The reduced operator is given by
\begin{equation*}
  \wt{L}_{\veps, ij}(\eta)
  = \sum_{\nu \in \ZZ} \wt{l}_{\veps, ij}(\nu, \eta) T^{\nu}_{\veps},
\end{equation*}
where the coefficients $\wt{l}_{\veps, ij}$ for $\eta \in \ZZ^{d-1}$
are given by $\wt{l}_{\veps, ij}(\nu, \eta) = \sum_{\mu} l_{ij}(\nu,
\mu) e^{\I 2\pi \veps \mu\cdot\eta}$.  Analogously, we define
$\wt{B}_{\veps, kj}$ as the reduced operator of $B_{\veps, kj}$. We
denote $\wt{B}^1_{\veps}$ the operators corresponding to the first $p$
boundary conditions and $\wt{B}^2_{\veps}$ for the remaining $q-p$
ones.  Similarly, we also denote the reduced operators for $L_{ij}$ as
\begin{equation*}
  \wt{L}_{ij}(\partial_x, \theta) = L_{ij}(\partial_x, \imath\theta),
  \qquad \text{for } \theta \in S^d,
\end{equation*}
and analogously for $B_{kj}$. For these reduced operators, we consider three types of eigensolutions,
which are defined as follows.
\begin{definition} We call an eigensolution of type I a nontrivial solution of
  \begin{equation*}
    (\wt{L}_{\veps}(\eta) w)(x_{\nu}) = 0\qquad \text{for some }\eta \neq 0,
  \end{equation*}
  satisfying (a) $(\wt{B}^1_{\veps}(\eta) w)(x_0) = 0$; (b)
  $(\wt{B}^2_{\veps}(\eta) w)(x_0) = 0$; (c) $w(x_{\nu}) \to 0 \
  \text{as}\ \nu \to \infty$.
\end{definition}
\begin{definition} We call an eigensolution of type II a nontrivial solution of
  \begin{equation*}
    (\wt{L}(\partial_x, \theta) w)(x) = 0\qquad \text{for some }
    \theta \in S^{d},
  \end{equation*}
  satisfying (a) $(\wt{B}^1(\partial_x, \theta) w)(x_0) = 0$; (b)
  $w(x) \to 0 \ \text{as}\ x \to \infty$.
\end{definition}
\begin{definition}\label{def:eigen3} We call an eigensolution of type III a
  nontrivial solution to
  \begin{equation*}
    (\wt{L}_{\veps}(0)w)(x) = 0,
  \end{equation*}
  satisfying (a) $(\wt{B}^2_{\veps}(0) w)(x_0) = 0$;
  (b) $w(x_{\nu}) \to 0 \ \text{as}\ \nu \to
    \infty$.
\end{definition}

\textit{Remark.}
  We remark that type I and type III eigensolutions are connected to
  the notion of surface phonon in the physics literature. The
  complementing conditions give a mathematical characterization of the
  surface phonon.
  On the other hand, the type II
  eigensolutions are surface waves at the boundary of the PDE.
\begin{definition}[\textsc{Complementing Boundary Condition}] The system
  \eqref{eq:Lsys} with boundary conditions \eqref{eq:Bsys} satisfies
  the \emph{Complementing Boundary Condition} if there are no eigensolutions of
  type I, II, or III.
\end{definition}

We state the following regularity estimate of the solution of elliptic
system with boundary conditions, which is an adaptation of
\cite{StrikwerdaWadeBube:1990}*{Theorem 3.1} to our setting.
\begin{theorem}\label{thm:regboundary}
  If $u$ is a solution to the system \eqref{eq:Lsys} with boundary
  conditions \eqref{eq:Bsys} and Assumptions \ref{assump:boundary},
  \ref{assump:boundweight}, and \ref{assump:supple} are satisfied,
  then the following regularity estimate holds for each $s$ with
  $\wb{\rho} \leq s < \rho^{\ast}$ and $\veps$ sufficiently small, if
  and only if, the Complementing Boundary Condition holds.
  \begin{equation*}
    \norm{u}_{\veps, \tau + s}^2
    \lesssim  \norm{f}_{\veps, s - \sigma}^2 + \norm{u}_{\veps, 0}^2,
  \end{equation*}
\end{theorem}
\subsection{Regularity estimate for hybrid schemes}\label{sec:regularity}

We will prove the regularity estimate for the the linearized operator $\mc{H}_{\qc} =
\mc{H}_{\qc}[0]$ of the hybrid scheme by
applying the regularity estimate Theorem~\ref{thm:regboundary}. 

\begin{theorem}[Regularity]\label{thm:regularity}
  Under Assumptions~\ref{assump:stabatom}, \ref{assump:boundarycond},
  and \ref{assump:complement}, for any $v \in H^2_{\veps}(\Omega)$, we
  have for $\veps$ sufficiently small
  \begin{equation}\label{eq:regularity}
    \norm{v}_{\veps, 2} \lesssim \norm{\mc{H}_{\qc} v}_{\veps, 0} + \norm{v}_{\veps, 0}.
  \end{equation}
\end{theorem}

\begin{proof}
  By a standard localization argument as in \cite{LaxNirenberg:1966},
  it suffices to show that for a solution $u$ of the system
  \eqref{eq:Lcorig}-\eqref{eq:Laorig}, the following regularity
  estimate holds for $\veps$ sufficiently small
  \begin{equation}\label{eq:reg}
    \norm{u}_{\veps, 2} \lesssim \norm{f}_{\veps, 0} + \norm{u}_{\veps, 0}.
  \end{equation}
  Recall that after folding with respect to the interface, the system
  \eqref{eq:Lcorig}-\eqref{eq:Laorig} is equivalent to the system
  \eqref{eq:L} with boundary conditions \eqref{eq:B}. Let us now
  verify that the system \eqref{eq:L}-\eqref{eq:B} satisfies the
  assumptions of Theorem~\ref{thm:regboundary}. First by construction
  of the operator $L$, we know that
  \begin{equation}\label{eq:productsymbol}
    \det l_{\veps, ij}(\zeta)=\det h_{\veps, ij}(\zeta)\det h_{\at, ij}(\zeta).
  \end{equation}
  The regular ellipticity of the operator $L_{\veps}$ is guaranteed by
  Assumption~\ref{assump:stabatom}, as $\mc{H}_{\at}$ and
  $\mc{H}_{\veps}$ are both regular elliptic by \cite[Lemma
  3.3]{LuMing:2011}.  Moreover, the order of $L_{\veps}$ is $(2, 0)$,
  and hence, $p = 2d$.  From \eqref{eq:productsymbol}, the
  Assumption~\ref{assump:boundary} for the system
  \eqref{eq:L}-\eqref{eq:B} is equivalent to
  Assumption~\ref{assump:boundarycond}.  By definition, we have
  \begin{equation*}
    \rho_k = \lfloor (k-1) / d \rfloor - 2.
  \end{equation*}
  Hence, $\overline{\rho} = 0$ and $\rho^{\ast} = 1$. Assumption
  \ref{assump:boundweight} is also satisfied by the system
  \eqref{eq:L}-\eqref{eq:B}.  The Supplementary Assumption
  \ref{assump:supple} is automatically satisfied since $2d \geq 3$ as
  a result in \cite{AgmonDouglisNirenberg:1959}.

  Applying Theorem \ref{thm:regboundary}, we have $\norm{U}_{\veps, 2}
  \lesssim \norm{F}_{\veps, 0} + \norm{U}_{\veps, 0}$.  By definition
  of $F$ and $U$, it is clear that $\norm{F}_{\veps, 0} =
  \norm{f}_{\veps, 0}$, $\norm{U}_{\veps, 0} \leq 2 \norm{u}_{\veps,
    0}$, and $\norm{u}_{\veps, 2} \leq \norm{U}_{\veps,
    2}$. Therefore, we arrive at the desired estimate \eqref{eq:reg}.
\end{proof}

%
\section{Stability}\label{sec:stability}

The main purpose of this section is to establish the following
stability estimate.
\begin{theorem}[Stability]\label{thm:stability}
  Under Assumptions~\ref{assump:stabatom}, \ref{assump:boundarycond},
  and \ref{assump:complement}, for any $v \in H^2_{\veps}(\Omega)$
  with mean zero, we have
  \begin{equation}\label{eq:stability}
    \norm{v}_{\veps,2} \lesssim \norm{\mc{H}_{\qc} v}_{\veps,0}
  \end{equation}
  for $\veps$ sufficiently small.
\end{theorem}

To obtain the stability estimate from the regularity estimate
Theorem~\ref{thm:regularity}, we need to eliminate $\norm{v}_{\veps,
  0}$ on the right hand side of \eqref{eq:regularity}. The proof is
based on the uniqueness of the continuous system from ellipticity, the
consistency of the finite difference schemes to the continuous system,
and the regularity estimate Theorem~\ref{thm:regularity}. We remark
that the stability analysis in our previous work \cite{LuMing:2011}
relies on the smoothness of $\varrho$, so it does not apply to the
sharp transition case. The current proof does not require smoothness
of $\varrho$ and so applies both to the hybrid method with smooth or
sharp transitions.

In order to connect the finite difference system with continuous PDE,
we need to extend grid functions on $\Omega_{\veps}$ to continuous
functions defined in $\Omega$. For this purpose, let us define an
interpolation operator $Q_{\veps}$ as follows. For any lattice
function $u$ on $\Omega_{\veps}$, we define $Q_{\veps} u \in
L^2(\Omega)$ as
\begin{equation}\label{eq:Qvepsdef}
  (Q_{\veps} u)(x) = \sum_{\xi \in \LL_{\veps}^{\ast}}
  e^{\I x \cdot \xi} \wh{u}(\xi), \quad x \in \Omega.
\end{equation}
By the Fourier inversion formula, we know that $Q_{\veps} u$ agrees
with $u$ on $\Omega_{\veps}$. We recall that by
\cite{LuMing:2011}*{Lemma 4.2},
\begin{equation}\label{eq:interpolation}
  c_k \norm{u}_{H^k_{\veps}(\Omega)} \leq
  \norm{Q_{\veps} u}_{H^k(\Omega)} \leq C_k \norm{u}_{H^k_{\veps}(\Omega)}, \quad \forall \; k \geq 0.
\end{equation}
Let $\chi$ be a standard nonnegative cut-off function on $\RR^d$,
which is smooth and compactly supported, with $\norm{\chi}_{L^1} = 1$.
Let $\chi_{\veps}$ be the scaled version $\chi_{\veps}(x) =
\veps^{-(\alpha d)} \chi( \veps^{-\alpha} x)$ for some $\alpha$ with
$0 < \alpha < 1$. The choice of the value of $\alpha$ will be
specified in the proof of Proposition~\ref{prop:contcons}.  Define a
low-pass filter operator $K_{\veps}$ for $f \in L^2(\Omega)$ using
$\wh{\chi_{\veps}}$ as Fourier multiplier:
\begin{equation*}
  \wh{K_{\veps} f}(\xi) = (2\pi)^{d} \wh{f}(\xi) \wh{\chi_{\veps}}(\xi)
  = (2\pi)^{d} \wh{f}(\xi) \wh{\chi}(\veps^{\alpha} \xi).
\end{equation*}
In real space, $K_{\veps}$ convolves $f$ with $\chi_{\veps}$.  Integrating by parts, we
obtain
\begin{align}
  & \label{eq:chidecay} \abs{\wh{\chi_{\veps}}(\xi)} \leq C_k
  \abs{\veps^{\alpha} \xi}^{-k}\quad \text{for all}\ k \in \ZZ_+, \\
  & \label{eq:chiunity} (2\pi)^{d} \wh{\chi_{\veps}}(0) = 1.
\end{align}
Hence, $K_{\veps}$ is indeed a low-pass filter.
For simplicity of notation, we denote
\begin{equation*}
  \wb{u}_{\veps} = K_{\veps} Q_{\veps} u_{\veps}
\end{equation*}
for any lattice function $u_{\veps}$ on $\Omega_{\veps}$.

We will use the following result on the consistency of the
symbols. It is an easy corollary from the consistency of the scheme, the details can be found in the Supplementary Materials Section~\ref{sec:cons}.
\begin{lemma}[Consistency of symbols of linearized
  operators]\label{lem:Hcons}
  There exists $\veps_0 > 0$ and $s > 0$ such that for any $\veps \leq
  \veps_0$, $x \in \Omega_{\veps}$, and $\xi \in \LL^{\ast}_{\veps}$,
  we have
  \begin{align}
    &\abs{ {h}_{\at}(\xi) - {h}_{\CB}(\xi) } \leq C \veps^2 (1
    + \abs{\xi}^2)^{s/2},\label{eq:HconsatCB}\\
    &\abs{ {h}_{\at}(\xi) - {h}_{\qc}(x, \xi)}\leq C \veps^2
    (1 + \abs{\xi}^2)^{s/2}. \label{eq:Hconsathy}
  \end{align}
\end{lemma}

The key element of the proof of Theorem~\ref{thm:stability} is the
following proposition.
\begin{proposition}\label{prop:contcons}
  For $\{v_{\veps}\}_{\veps > 0}$ such that $v_{\veps}\in
  H^2_{\veps}(\Omega)$ and $\norm{v_{\veps}}_{\veps, 2}$ is uniformly
  bounded, we have
  $\lim_{\veps \to 0+} \norm{\mc{H}_{\CB} \wb{v}_{\veps}-
      \wb{\mc{H}_{\qc} v_{\veps}}}_{L^2(\Omega)} = 0$.
\end{proposition}

\begin{proof}
  By triangular inequality,
  \begin{multline*}
    \norm{\mc{H}_{\CB} \wb{v}_{\veps}- \wb{\mc{H}_{\qc}
        v_{\veps}}}_{L^2(\Omega)} \leq \norm{\overline{\mc{H}_{\qc}
        \wb{v}_{\veps}} - \wb{\mc{H}_{\qc} v_{\veps}}}_{L^2(\Omega)} +
    \norm{\overline{\mc{H}_{\at} \wb{v}_{\veps}} - \overline{\mc{H}_{\qc}
        \wb{v}_{\veps}}}_{L^2(\Omega)} \\
    + \norm{\mc{H}_{\CB} \wb{v}_{\veps} - \overline{\mc{H}_{\at}
        \wb{v}_{\veps}}}_{L^2(\Omega)}.
  \end{multline*}
  Hence, it suffices to show each term on the right-hand side goes to
  zero as $\veps \to 0$.
  \smallskip

  \noindent\emph{Step 1.}
  By \eqref{eq:interpolation} and the definition of
  $\mc{H}_{\qc}$, we have
  \begin{equation*}
    \begin{aligned}
      \norm{\overline{\mc{H}_{\qc} \wb{v}_{\veps}} - \wb{\mc{H}_{\qc}
          v_{\veps}}}_{L^2(\Omega)} & \leq \norm{\mc{H}_{\qc}
        \wb{v}_{\veps} - \mc{H}_{\qc} v_{\veps}}_{\veps, 0} \\
      & \leq \norm{\mc{H}_{\at} \wb{v}_{\veps} - \mc{H}_{\at}
        v_{\veps}}_{\veps, 0} + \norm{\mc{H}_{\veps} \wb{v}_{\veps} -
        \mc{H}_{\veps} v_{\veps}}_{\veps, 0}.
    \end{aligned}
  \end{equation*}
  We now show that $\norm{\mc{H}_{\veps} \wb{v}_{\veps} -
    \mc{H}_{\veps} v_{\veps}}_{\veps, 0}$ converges to zero as $\veps
  \to 0$, the argument for the other term is identical. By Parseval's
  identity,
  \begin{equation}\label{eq:chiminusone}
    \norm{\mc{H}_{\veps} \wb{v}_{\veps} -
      \mc{H}_{\veps} v_{\veps}}_{\veps, 0} =  \norm{{h}_{\veps}(\xi)
      (\wh{\chi_{\veps}}(\xi) - 1) \wh{v}_{\veps}(\xi)}_{l^2(\LL_{\veps}^{\ast})}.
  \end{equation}
  Note that $
  \norm{{h}_{\veps}(\xi)\wh{v}_{\veps}(\xi)}_{l^2(\LL_{\veps}^{\ast})}
  = \norm{\mc{H}_{\veps} v_{\veps}}_{\veps, 0}$ is bounded as
  $v_{\veps} \in H^2_{\veps}(\Omega)$ and that $\wh{\chi_{\veps}}(\xi)
  \to 1$ for any $\xi \in \LL_{\veps}^{\ast}$, as $\veps \to 0$. By
  dominance convergence, the right-hand side of \eqref{eq:chiminusone}
  converges to zero in the limit. Therefore, $\norm{\overline{\mc{H}_{\qc}
      \wb{v}_{\veps}} - \wb{\mc{H}_{\qc} v_{\veps}}}_{L^2(\Omega)} \to 0$.

  \smallskip
  \noindent\emph{Step 2.} By \eqref{eq:interpolation}, we have
  \begin{equation*}
    \norm{\overline{\mc{H}_{\at} \wb{v}_{\veps}} - \overline{\mc{H}_{\qc}
        \wb{v}_{\veps}}}_{L^2(\Omega)} \lesssim
    \norm{\mc{H}_{\at} \wb{v}_{\veps} - \mc{H}_{\qc}
        \wb{v}_{\veps}}_{\veps, 0}.
  \end{equation*}
  Hence it suffices to estimate the right-hand side. We calculate
  \begin{equation*}
    \mc{H}_{\at} \wb{v}_{\veps} - \mc{H}_{\qc}\wb{v}_{\veps} =  \sum_{\xi \in \LL_{\veps}^{\ast}} ({h}_{\at}(\xi)
    - {h}_{\qc}(x, \xi)) \wh{\wb{v}}_{\veps}(\xi) e^{\I x \cdot
      \xi}.
  \end{equation*}
  Therefore, using \eqref{eq:Hconsathy} in Lemma~\ref{lem:Hcons}, we have
  \begin{equation*}
    \norm{\mc{H}_{\at} \wb{v}_{\veps} - \mc{H}_{\qc}
      \wb{v}_{\veps}}_{\veps, 0}^2 \lesssim
    \sum_{\xi \in \LL_{\veps}^{\ast}} \veps^4 (1 + \abs{\xi}^2)^{s}
    {\bigl\vert\wh{\wb{v}}_{\veps}(\xi)\bigr\vert}^2.
  \end{equation*}
  To estimate the right-hand side, note that by \eqref{eq:chidecay} we
  have for any $k \in \ZZ_+$,
  \begin{equation*}
    \abs{\wh{\wb{v}}_{\veps}(\xi)}^2 = \abs{\wh{\chi_{\veps}}(\xi)}^2
      \abs{\wh{v}_{\veps}(\xi)}^2 \leq C_k \abs{\veps^{\alpha} \xi}^{-2k}
      \abs{\wh{v}_{\veps}(\xi)}^2.
  \end{equation*}
  Taking $k$ sufficiently large so that $(1 + \abs{\xi}^2)^s
  \abs{\xi}^{-2k}$ is bounded for $\xi \in \LL^{\ast} \backslash
  \{0\}$, we then have
  \begin{equation*}
    \norm{\mc{H}_{\at} \wb{v}_{\veps} - \mc{H}_{\qc}
      \wb{v}_{\veps}}_{\veps, 0}^2 \lesssim \veps^4 \veps^{-2 \alpha k}
    \sum_{\xi \in \LL_{\veps}^{\ast}} \abs{\wh{v}_{\veps}(\xi)}^2.
  \end{equation*}
  Choose $\alpha$ in the low-pass filter such that $\alpha < 2 / k$,
  we get $\norm{\overline{\mc{H}_{\at} \wb{v}_{\veps}} -
    \overline{\mc{H}_{\qc} \wb{v}_{\veps}}}_{L^2(\Omega)} \to 0$ as
  $\veps\to 0+$.

  \smallskip
  \noindent\emph{Step 3.} By definition,
  $\wh{\mc{H}_{\CB} \wb{v}_{\veps}}(\xi)
    = {h}_{\CB}(\xi) \wh{\chi}(\veps^{\alpha} \xi) \wh{v}_{\veps}(\xi)$.
  For the discrete system, we have
  \begin{equation*}
    \begin{aligned}
      \wh{\wb{\mc{H}_{\at} v_{\veps}}}(\xi) & =
      \wh{\chi}(\veps^{\alpha} \xi) \Bigl( \frac{\veps}{2\pi} \Bigr)^d
      \sum_{x \in \Omega_{\veps}} e^{-\I \xi \cdot x} \sum_{\eta \in
        \LL^{\ast}_{\veps}} e^{\I x \cdot
        \eta}{h}_{\at}(\eta) \wh{v}_{\veps}(\eta) \\
      & = \wh{\chi}(\veps^{\alpha} \xi) {h}_{\at}(\xi)
      \wh{v_{\veps}}(\xi).
    \end{aligned}
  \end{equation*}
  Hence, we have, combined with \eqref{eq:HconsatCB} and
  \eqref{eq:chidecay}, for some $s > 0$ and any $k \in \ZZ_+$,
  \begin{equation*}
    \begin{aligned}
      \norm{\mc{H}_{\CB} \wb{v}_{\veps} - \wb{\mc{H}_{\at}
          v_{\veps}}}_{L^2(\Omega)}^2 & = \sum_{\xi \in
        \LL_{\veps}^{\ast}} ({h}_{\at}(\xi) - {h}_{\CB}(\xi))^2
      \abs{\wh{\chi}(\veps^{\alpha} \xi)}^2
      \abs{\wh{v}_{\veps}(\xi)}^2 \\
      & \leq C_k \veps^4 \sum_{\xi \in \LL_{\veps}^{\ast}} (1 +
      \abs{\xi})^{s} \abs{\veps^{\alpha} \xi}^{-2k}
      \abs{\wh{v}_{\veps}(\xi)}^2.
    \end{aligned}
  \end{equation*}
  We choose a sufficiently large $k$, so that the right-hand side is
  bounded by $C \veps^{4 - 2 \alpha k} \norm{v_{\veps}}_{\veps,
    0}^2$. The choice of $\alpha < 2/k$ as above guarantees that
  $\norm{\mc{H}_{\CB} \wb{v}_{\veps} - \wb{\mc{H}_{\at}
      v_{\veps}}}_{L^2(\Omega)} \to 0$ as $\veps \to 0$.
\end{proof}

\smallskip

\begin{proof}[Proof of Theorem~\ref{thm:stability}]
  Suppose \eqref{eq:stability} does not hold, then there is a sequence
  of functions $\{w_k\}$ and $\veps_k > 0$ such that $\lim_{k \to
    \infty} \norm{w_k}_{\veps_k, 2} \to \infty$, and
  $\norm{\mc{H}_{\qc} w_k}_{\veps_k, 0} \leq c$ and $\sum_{x\in
    \Omega_{\veps_k}} w_k(x) = 0$ for all $k$.  Set $v_k = w_k /
  \norm{w_k}_{\veps_k, 2}$, we then have
  \begin{equation*}
    \norm{v_k}_{\veps_k, 2} = 1 \quad \text{and} \quad  \sum_{x\in \Omega_{\veps_k}} v_k(x) = 0 \qquad \text{for all } k; \qquad
    \lim_{k \to \infty} \norm{\mc{H}_{\qc} v_k}_{\veps_k, 0} \to 0.
  \end{equation*}
  Note that $\mc{H}_{\CB} \wb{v}_k = \wb{\mc{H}_{\qc} v_k} + (
  \mc{H}_{\CB} \wb{v}_k - \wb{\mc{H}_{\qc} v_k})$.  Since
  $\norm{\mc{H}_{\qc} v_k}_{\veps_k, 0} \to 0$, we have $
  \norm{\wb{\mc{H}_{\qc} v_k}}_{L^2(\Omega)} \to 0$, as $k \to
  \infty$.  Moreover, by Proposition~\ref{prop:contcons},
  \begin{equation*}
    \norm{\mc{H}_{\CB} \wb{v}_k -
      \wb{\mc{H}_{\qc} v_k}}_{L^2(\Omega)} \to 0 \quad \text{as } k \to \infty.
  \end{equation*}
  Hence $\norm{\mc{H}_{\CB} \wb{v}_k }_{L^2(\Omega)} \to 0$. Note also
  that the average of $\wb{v}_k$ is zero, since $\wh{\wb{v}_k}(0) =
  0$. By the invertibility of $\mc{H}_{\CB}$ on the subspace
  orthogonal to constant function, $\norm{\wb{v}_k}_{L^2(\Omega)} \to
  0$, as $ k \to \infty$, while $\norm{v_k}_{\veps_k, 2} = 1$. It
  follows then $\norm{v_k}_{\veps_k, 0} \to 0$. Indeed, since
  \begin{equation*}
    \norm{v_k}_{\veps_k, 1} = \sum_{\xi \in \LL^{\ast}_{\veps_k}}
    \Lambda^2_{\veps_k}(\xi) \abs{\wh{v}_k(\xi)}^2 \leq 1,
  \end{equation*}
  for any $\delta > 0$, there exist $\Xi>0$ and $k_1$, such that for
  any $k \geq k_1$,
  \begin{equation}\label{eq:k1}
    \sum_{\xi \in \LL^{\ast}_{\veps_k},\, \abs{\xi} \geq \Xi}
    \abs{\wh{v}_k(\xi)}^2 < \delta / 2.
  \end{equation}
  On the other hand, due to \eqref{eq:chiunity}, there exists $k_2$,
  such that for $k \geq k_2$
  \begin{equation}\label{eq:k2}
    \sum_{\xi \in \LL^{\ast}_{\veps_k}, \, \abs{\xi} < \Xi}
    \Bigl\lvert \abs{\wh{v}_k(\xi)^2} - \abs{\wh{\wb{v}}_k(\xi)}^2
    \Bigr\rvert \leq \delta / 4.
  \end{equation}
  Moreover, as $\norm{\wb{v}_k}_{L^2} \to 0$, there exists $k_3$, such
  that for $k \geq k_3$,
  \begin{equation}\label{eq:k3}
    \sum_{\xi \in \LL^{\ast}_{\veps_k}, \abs{\xi} < \Xi}
    \abs{\wh{\wb{v}}_k(\xi)}^2 \leq \delta / 4.
  \end{equation}
  Combining~\eqref{eq:k1}--\eqref{eq:k3} together, we have for $k \geq
  \max(k_1, k_2, k_3)$,
  \begin{equation*}
    \norm{v_k}_{\veps_k, 0}^2 = \sum_{\xi \in \LL^{\ast}_{\veps_k}}
    \abs{\wh{v}_k}^2\leq \delta.
  \end{equation*}
  Hence, $\lim_{k \to \infty} \norm{v_k}_{\veps_k, 0} = 0$. From
  Theorem~\ref{thm:regularity}, this implies
  $\lim_{k\to \infty} \norm{v_k}_{\veps_k, 2} = 0$.
  The contradiction with the choice of $v_k$ proves the Theorem.
\end{proof}


%
\section{Example}\label{sec:example}

We apply the result here to some concrete examples of
atomistic-to-continuum hybrid method. More specifically, we check the
assumptions of Theorem~\ref{thm:stability} which then imply the
stability of the scheme near the ground state.

\subsection{Example 1. Triangular lattice with harmonic interaction} Consider a force-based method on
a triangular lattice.  Figure~\ref{schematic} gives the geometry and
the choice of the atomistic and continuum regions. The interface
between the atomistic and continuum regions is parallel to the $(1, \sqrt{3})/2$
direction of the lattice. The interaction between atoms is assumed to
be harmonic (quadratic potential). The force balance equation reads
\begin{equation}\label{eq:trihybrid}
\Pi_{\veps} \mc{F}_{\qc}[z](x)=f_\veps(x),
\end{equation}
with $\mc{F}_{\qc}$ given by $\mc{F}_{\at}$ in $\Omega_a$ and by
$\mc{F}_{\veps}$ in $\Omega_c$, where
\[
\mc{F}_{\at}[z](x){:}=\dfrac1{\veps}\sum_{i=1}^{12}D_{\veps,\mu_i}^+z(x),
\qquad\mc{F}_{\veps}[z](x){:}=\dfrac4{\veps}\sum_{i=1}^6D_{\veps,\mu_i}^+z(x),
\]
where $\{\mu_i\}_{i=1}^6$ and $\{\mu_i\}_{i=7}^{12}$ are the first and
the second neighborhood interaction ranges of the triangular lattice
$\LL$, respectively.
\begin{figure}
\centering
\includegraphics[width=2.4in]{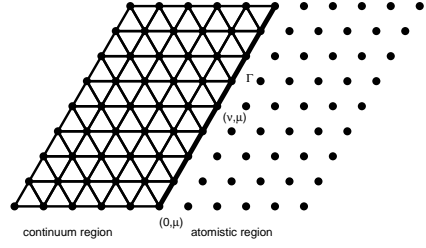}
\caption{\small Illustration of the geometry of the hybrid method. The
  coordinate of each atom is $(\nu,\mu)$, and the interface $\Gamma$
  along the atoms labeled with $(0,\mu)$. The interface $\Gamma$
  divides the domain into the atomistic region $\Omega_a$ and the
  continuum region $\Omega_c$.}\label{schematic}
\end{figure}

We rewrite the coupled force balance equation into a system. For any
$x=ma_1+na_2\in\Omega_a$, we define the map $\wt{x}= -m a_1+na_2$, and
denote $z(x)=\bigl(y(x-2a_1),y(\wt{x})\bigr)$, where $a_1$ and $a_2$
are the basis vectors of the triangular lattice $\LL$. The hybrid
method can be rewritten as
\begin{equation}\label{eq:equilibrium}
\mc{F}_{\qc}[z](x)=\ell_{\veps}(x),
\end{equation}
where
$\mc{F}_{\qc}=\bigl(\begin{smallmatrix}
\mc{F}_{\at}&0\\
0&\mc{F}_{\veps}
\end{smallmatrix}\bigr)$ and $\ell_{\veps}(x)=\bigl(f_{\veps}(x-2a_1),f_{\veps}(\wt{x})\bigr)$.
It is supplemented
with the following boundary conditions:
\begin{equation}\label{eq:bdexample}
\sum_{j=1}^2B_{\veps,kj}z_j(x)=0,\qquad x\in\Gamma, k=1,2,3,
\end{equation}
where $B_{\veps,k1}=\Lr{D_{\veps,\mu_1}^+}^{k-1}$, and
$B_{\veps,k2}=-\Lr{D_{\veps,\mu_1}^+}^{k-1}T_{\veps}^{-2\mu_1}$.

Applying Fourier transform in the tangential direction $a_2$, we obtain
\begin{align*}
  \wt{\mc{F}}_{\veps}&=\dfrac4{\veps^2}\Lr{T_{\veps}^{\mu_1}(1+e^{-\I\xi\cdot
      \mu_2}) +T_{\veps}^{-\mu_1}(1+e^{\I\xi\cdot\mu_2})+2\cos(\xi\cdot\mu_2)-6},\\
  \wt{\mc{F}}_{\at}&=\dfrac1{\veps^2}\Bigl(T_{\veps}^{\mu_1}(1+e^{-\I\xi\cdot
    \mu_2} +e^{\I\xi\cdot\mu_2}+e^{-2\I\xi\cdot\mu_2})
  +T_{\veps}^{-\mu_1}(1+e^{\I\xi\cdot\mu_2}+e^{-\I\xi\cdot\mu_2}+e^{2\I\xi\cdot \mu_2})\\
  &\phantom{\dfrac1{\veps^2}\Bigl(}\qquad+T_{\veps}^{2\mu_1}e^{-\I\xi\cdot
    \mu_2}+T_{\veps}^{-2\mu_1}e^{\I\xi\cdot\mu_2}+2\cos(\xi\cdot\mu_2)-12\Bigr).
\end{align*}

The first step is to consider the distribution of the
roots of the following two characterization equations. Let
$\zeta=e^{\I\xi\cdot\mu_2}=e^{\I\theta}$, obviously $\zeta\not=1$
unless $\theta=0$.
\begin{equation}
z(1+\wb{\zeta})+z^{-1}(1+\zeta)+\zeta+\wb{\zeta}-6=0,\label{eq:cha1}
\end{equation}
and
\begin{equation}\label{eq:cha2}
  \begin{aligned}
    z^2\wb{\zeta}+z^{-2}\zeta&+z(1+\zeta+\wb{\zeta}+\wb{\zeta^2})\\
    &+z^{-1}(1+\zeta+\wb{\zeta}+\zeta^2)+\zeta+\wb{\zeta}-12=0.
  \end{aligned}
\end{equation}
The equation~\eqref{eq:cha1} has two roots $z_1$ and $z_2$ that satisfy
$\abs{z_1z_2}=\abs{\zeta}=1$. We cannot have $\abs{z_1}=\abs{z_2}=1$,
otherwise $\det\wt{\mc{F}}_{\veps}=0$, which contradicts with the bulk
stability condition. Therefore, we have two distinct roots, one inside
the unit disk, the other outside. In particular, we denote the root inside
the unit disk as $z_1$.

We let $z=w\zeta^{1/2}$, and
write~\eqref {eq:cha2} as
\begin{equation}\label{eq:cha21}
w^2+w^{-2}+(\zeta^{1/2}+\wb{\zeta}^{1/2}+\zeta^{3/2}
+\wb{\zeta}^{3/2})(w+w^{-1})
+\zeta+\zeta^{-1}-12=0.
\end{equation}
Let $s=w+w^{-1}$,
$A=\zeta^{1/2}+\wb{\zeta}^{1/2}+\zeta^{3/2}+\wb{\zeta}^{3/2}$ and
$f(s)=s^2+As+\zeta+\wb{\zeta}-14$, the above equation can be written
as $f(s)=0$.  A direct calculation 
shows that $f(2) < 0$ and $f(-2)<0$, which implies that there exist
two roots $s_1$ and $s_2$ with $s_1>2$ and $s_2<-2$. This yields that
\eqref{eq:cha21} has four roots $\{w_i\}_{i=1}^4$ satisfying $w_1>1,
w_2<1,-1<w_3<0$ and $w_4<-1$. 
To sum up, there exists four distinct roots for~\eqref {eq:cha2} such
that two inside the unit disk while the other two outside the unit
disk. In particular, $\abs{z_2},\abs{z_3}<1$.
Therefore, we require three boundary conditions, which is
consistent with~\eqref{eq:bdexample}.

We next verify Assumption~\ref{assump:complement} for the coupling
scheme.
For mode I, we have the following form of the solution
\[
z(x_i,\zeta)=c_1 \bigl(\begin{smallmatrix}1\\
0\end{smallmatrix}\bigr)z_1^i+c_2\bigl(\begin{smallmatrix}0\\
1\end{smallmatrix}\bigr)z_2^i+c_3\bigl(\begin{smallmatrix}0\\
1\end{smallmatrix}\bigr)z_3^i,\qquad\abs{z_k}<1,\quad k=1,2,3.
\]
As $\zeta\to 1$, we have
$\abs{z_1},\abs{z_2}\to 1$ while $\abs{z_3}\to 3-2\sqrt2$. Substituting
the above expression into the boundary conditions, we obtain
\[
\begin{pmatrix}
1&-1&-1\\
z_1&-z_2&-z_3\\
z_1^{-1}&-z_2^{-1}&-z_3^{-1}
\end{pmatrix}
\begin{pmatrix}c_1\\
c_2\\
c_3\end{pmatrix}=\begin{pmatrix}0\\
0\\0\end{pmatrix}.
\]
The determinant of the matrix is nonzero since $z_1$, $z_2$ and $z_3$
are distinct, the details can be found in the Supplementary Materials
Lemma~\ref{lem:root}. Hence we conclude that there does not exist
mode I eigenfunction.

For mode II, notice that by definition, we have
\[
\Lr{\wb{\mc{H}}_{\text{CB}}(\partial_x,\theta)w}(x)=6(\partial_x^2-\theta^2)w(x)=0,
\]
it is clear the only solution satisfying $w(x)\to 0$ as
$x\to\pm\infty$ and the boundary condition is the trivial solution $w(x)\equiv 0$, hence mode II eigenfunction does not exist.

For mode III, we have $\zeta\to 1$ as $\theta\to 0$. The
solution takes the form
\[
\lim_{\theta\to 0}z(x_i,\theta)=c_3(0)\bigl(\begin{smallmatrix}0\\
  1\end{smallmatrix}\bigr)z_3^i\quad\text{with}\quad z_3=2\sqrt2-3,
\]
substituting the above expression into the boundary condition, we obtain
\[
c_3(0)\bigl(\begin{smallmatrix}0\\
1\end{smallmatrix}\bigr)z_3^{-1}=c_3(0)\bigl(\begin{smallmatrix}0\\
1\end{smallmatrix}\bigr)z_3,
\]
which yields $c_3(0)=0$. This concludes that there does not exist mode III
eigenfunction. Therefore, the coupling scheme is stable and convergent.

\subsection{Example 2. Triangular lattice with truncated Lennard-Jones potential}
For the second example, we consider an atomistic-to-continuum coupling
with the same geometry as Example 1; but now the atoms are interacted
with Lennard-Jones potential, truncated to the second nearest neighbor
interactions. The (linearized) force balance equation has the same
form as~\eqref{eq:trihybrid}, with
\begin{align*}
& \begin{aligned}
\mc{F}_{\at}[z](x)&=\dfrac{2\ka_1}{\veps}\sum_{i=1}^6
D_{\veps,\mu_i}^+z(x)+\dfrac{2\ka_2}{\veps}
\sum_{i=1}^6\dual{D_{\veps,\mu_i}^+z(x)}{\mu_i}\mu_i\nn\\
&\quad+\dfrac{2\ka_3}{\veps}\sum_{i=7}^{12}D_{\veps,\mu_i}^+z(x)
+\dfrac{2\ka_4}{\veps}
\sum_{i=7}^{12}\dual{D_{\veps,\mu_i}^+z(x)}{\mu_i}\mu_i, \quad \text{and,}
\end{aligned} \\
& \mc{F}_\veps[z](x)=\frac{2(\ka_2+9\ka_4)}{\veps}\sum_{i=1}^6
\dual{D_{\veps,\mu_i}^+z(x)}{\mu_i}\mu_i,
\end{align*}
where $\ka$'s are defined as
\begin{align*}
  & \ka_1=g(1),\quad\ka_2=h(1),\quad\ka_3=g(\sqrt3),\quad\ka_4=h(\sqrt3), \\
  & g(r){:}=12K(-Kr^{-14}+r^{-8}),\quad
  h(r){:}=12K(14Kr^{-14}-8r^{-8})
\end{align*}
with $K=(1+3^{-3})/(1+3^{-6})$ and $\veps=(2/K)^{1/6}\sigma$.
For this example, we are no longer able to check stability by hand. Hence, we will combine with numerical computation to check the stability conditions.

Applying Fourier transform in the tangential direction, we obtain
\[
\wt{\mc{F}}_{\at}=2\veps^{-2}\mc{M}^{\at},\quad
\wt{\mc{F}}_{\veps}=2(\ka_2+9\ka_4)\veps^{-2}\mc{M}^{\CB}.
\]
For brevity, we only give the explicit expression of $\mc{M}^{\CB}$, which is a $2 \times 2$ matrix with
\[
\begin{aligned}
\mc{M}^{\CB}_{11}&=(T_{\veps}^{\mu_1}+T_{\veps}^{-\mu_1})
+\frac14(\zeta+\overline{\zeta}+\zeta T_{\veps}^{-\mu_1}+\overline{\zeta}T_{\veps}^{\mu_1})-3,\\
\mc{M}^{\CB}_{12}&=\mc{M}^{\CB}_{21}=\frac{\sqrt{3}}{4}(\zeta+\overline{\zeta}-\zeta T_{\veps}^{-\mu_1}-\overline{\zeta}T_{\veps}^{\mu_1}),\\
\mc{M}^{\CB}_{22}&=\frac34(\zeta+\overline{\zeta}+\zeta T_{\veps}^{-\mu_1}+\overline{\zeta}T_{\veps}^{\mu_1})-3.
\end{aligned}
\]

The first step is to consider the distribution of the roots of the two characteristic equations: $\det\mc{M^{\at}}=0$ and $\det\mc{M}^{\CB}=0$.
As to the characteristic equation $\det\mc{M^{\CB}}=0$, it is clear to
see there are four roots, two roots inside the unit disk, which are
denoted by $z_1$ and $z_2$. The remaining two roots are $\zeta/z_1$ and
$\zeta/z_2$.  The characteristic equation $\det\mc{M^{\at}}=0$ has eight
roots: Four are inside the unit disk, which are denoted by
$\{z_i\}_{i=3}^6$. The remaining four roots are $z_7=1/\wb{z}_3,
z_8=1/\wb{z}_4, z_9=\zeta/z_5$, and $z_{10}=\zeta/z_6$.  It may be checked
numerically that the roots are distinct.  Therefore, we need six
boundary conditions in total, which is consistent
with~\eqref{eq:bdexample}.

%

We next verify Assumption C. For mode I, we have the following form of
the solution: $z(x_i,\zeta)=\Lr{z_1(x_i,\zeta),z_2(x_i,\zeta)}$ with
\begin{align*}
z_1(x_i,\zeta) =\sum_{k=3}^{6}c_kz_k^i\mc{M}_2^{\at}(z_k,\zeta), \quad
z_2(x_i,\zeta) =c_1z_1^i\mc{M}_2^{\CB}(z_1,\zeta)
+c_2z_2^i\mc{M}_2^{\CB}(z_2,\zeta),
\end{align*}
where $\mc{M}_2^{\at}{:}=(\mc{M}^{\at}_{22},-\mc{M}^{\at}_{21})^{\TT}$
and $\mc{M}_2^{\CB}{:}=(\mc{M}^{\CB}_{22},-\mc{M}^{\CB}_{12})^{\TT}$.
Substituting the above expressions into the boundary condition, we
obtain $A(\zeta) \bd{c}(\zeta) = 0$, where $\bd{c}(\zeta) =
(c_1(\zeta), c_2(\zeta), \ldots, c_6(\zeta))^{\TT}$ and $A(\zeta)$ is
given by
\[\scriptsize
\begin{pmatrix}-\mc{M}^{\CB}_{22}(z_1)&-\mc{M}^{\CB}_{22}(z_2)
&\mc{M}^{\at}_{22}(z_3)
&\mc{M}^{\at}_{22}(z_4)&\mc{M}^{\at}_{22}(z_5)&\mc{M}^{\at}_{22}(z_6)\\
-\mc{M}^{\CB}_{12}(z_1)&-\mc{M}^{\CB}_{12}(z_2)&\mc{M}^{\at}_{12}(z_3)
&\mc{M}^{\at}_{12}(z_4)&\mc{M}^{\at}_{12}(z_5)&\mc{M}^{\at}_{12}(z_6)\\
-z_1\mc{M}^{\CB}_{22}(z_1)&-z_2\mc{M}^{\CB}_{22}(z_2)&z_3^{-1}\mc{M}^{\at}_{22}(z_3)
&z_4^{-1}\mc{M}^{\at}_{22}(z_4)&z_5^{-1}\mc{M}^{\at}_{22}(z_5)&z_6^{-1}
\mc{M}^{\at}_{22}(z_6)\\
-z_1\mc{M}^{\CB}_{12}(z_1)&-z_2\mc{M}^{\CB}_{12}(z_2)&z_3^{-1}
\mc{M}^{\at}_{12}(z_3)
&z_4^{-1}\mc{M}^{\at}_{12}(z_4)&z_5^{-1}\mc{M}^{\at}_{12}(z_5)&z_6^{-1}
\mc{M}^{\at}_{12}(z_6)\\
-z_1^{-1}\mc{M}^{\CB}_{22}(z_1)&-z_2^{-1}\mc{M}^{\CB}_{22}(z_2)&z_3
\mc{M}^{\at}_{22}(z_3)
&z_4\mc{M}^{\at}_{22}(z_4)&z_5\mc{M}^{\at}_{22}(z_5)&z_6\mc{M}^{\at}_{22}(z_6)\\
-z_1^{-1}\mc{M}^{\CB}_{12}(z_1)&-z_2^{-1}\mc{M}^{\CB}_{12}(z_2)
&z_3\mc{M}^{\at}_{12}(z_3)
&z_4\mc{M}^{\at}_{12}(z_4)&z_5\mc{M}^{\at}_{12}(z_5)
&z_6\mc{M}^{\at}_{12}(z_6)\end{pmatrix}.
\]
Let $\zeta=\exp(\imath\theta)$, and we take $\zeta_i=\exp(\imath\theta_i)$ with
$\theta_i=2i\pi/M$ for $i=1,\cdots,M$ with $M=1000$. Using
the high precision toolbox in Matlab, we obtain
\[
\min_{1\le i\le M}\abs{\det A(\zeta_i)}=0.00697689943572617629110144\\
4178785040926929.
\]
The profile for $\det A(\zeta)$ may be found in the left panel of
Fig.~\ref{Fig:det}, which clearly shows that $\abs{\det A(\zeta)}$ is
symmetric with respect to $\pi$ and is increasing over $(0,\pi)$. The
symmetry of $\det A(\zeta)$ can be proven by the distribution of the
roots. The right panel of Fig.~\ref{Fig:det} shows the finite
difference approximation
\[
D\abs{\det A(\zeta_j)}{:}=M\Lr{\abs{\det A(\zeta_{j+1})}- \abs{\det
    A(\zeta_{j})}}
\]
of the derivative of $\det A(\zeta)$. It is observed that $\det A$ is
strictly increasing, which indicates that $\det A(\zeta)\not=0$ if
$\zeta\not=0$ as $\det A(0)=0$. Hence no mode I eigenfunction exists.
\begin{figure}
  \centering
  \includegraphics[width=2.5in]{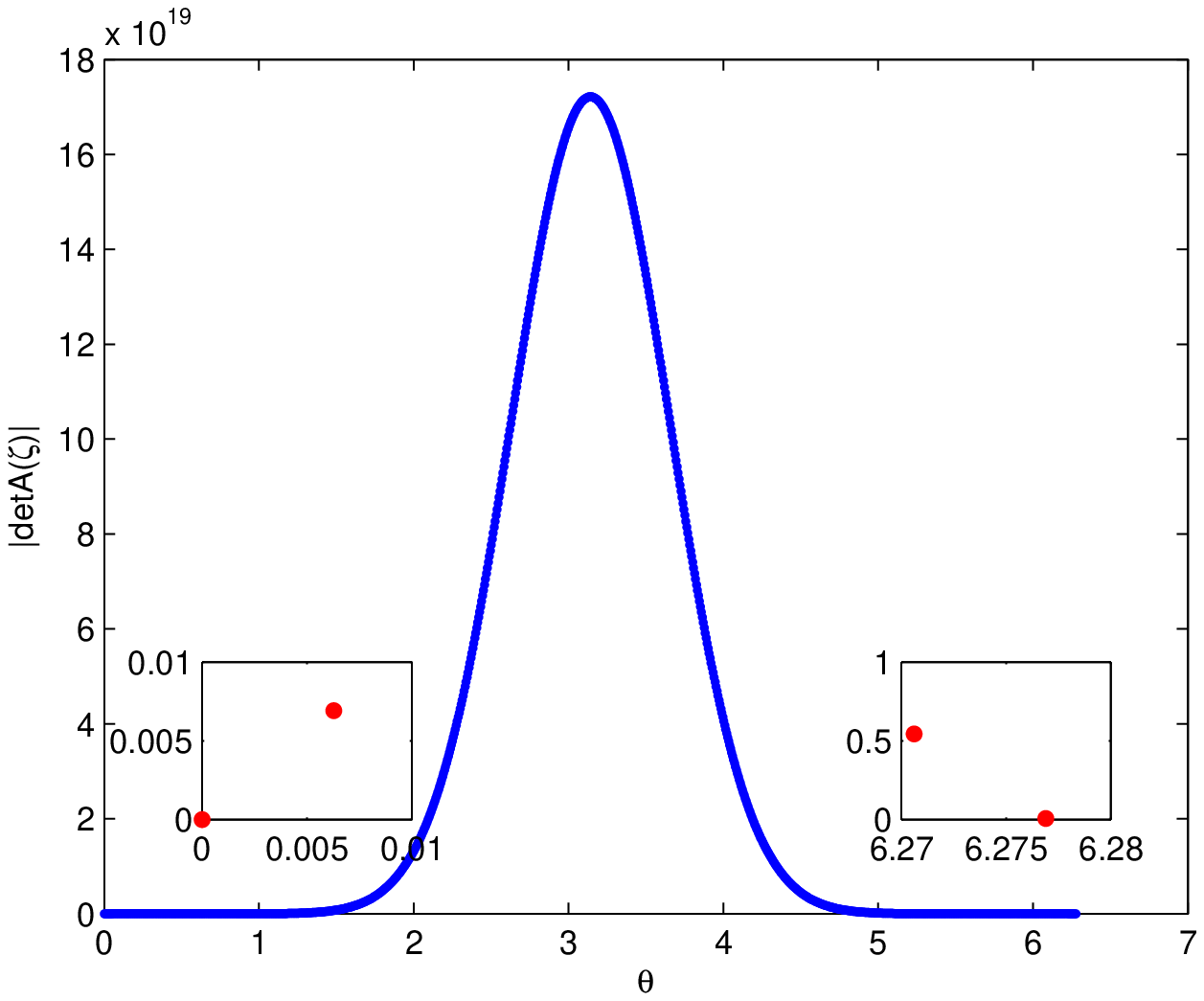}
  \includegraphics[width=2.5in]{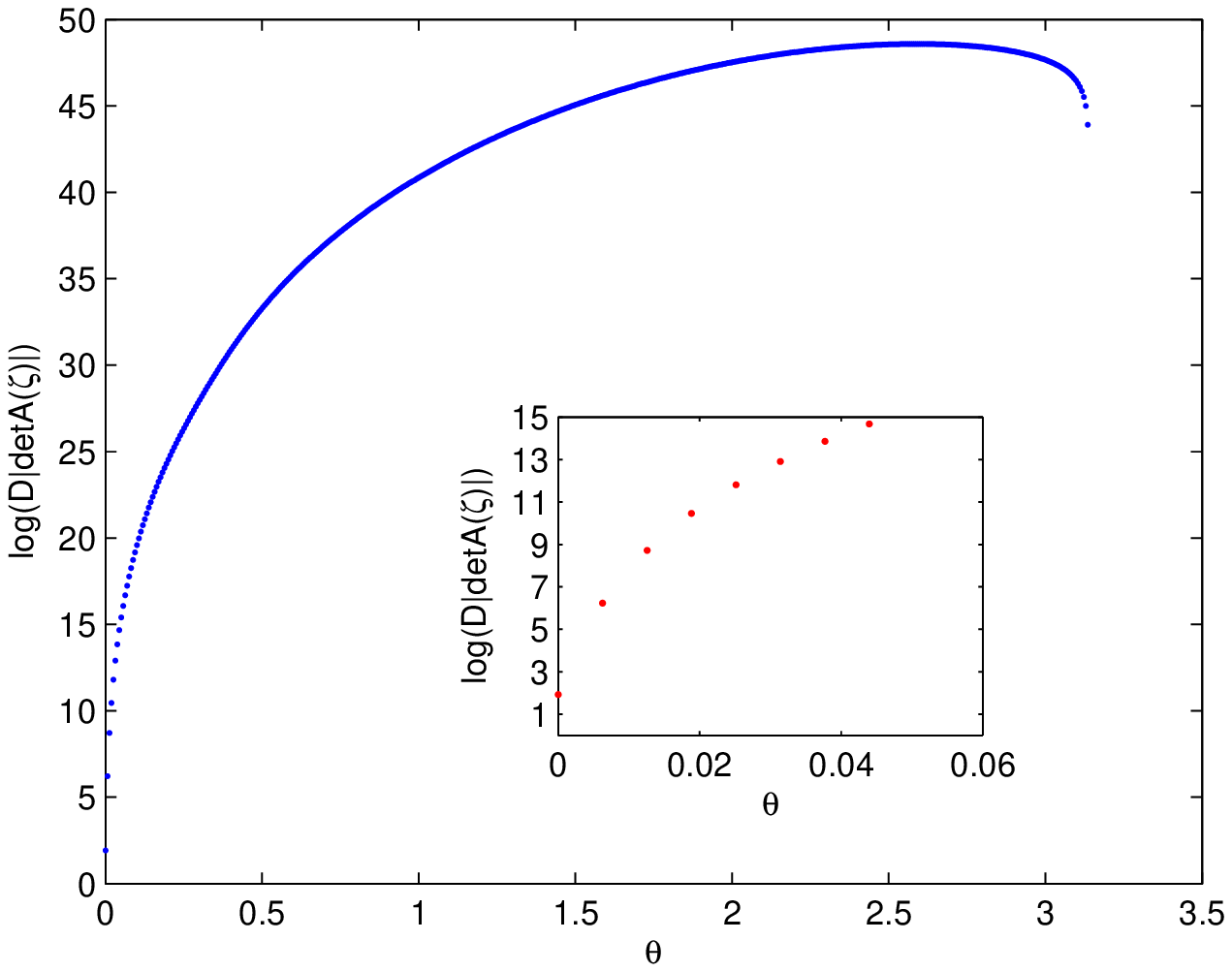}
  \caption{\small (left) $\abs{\det A(\zeta)}$ with $M =
    1000$; (right) $\log D\abs{\det A(\zeta_i)}$. The insets zoom in
    near the end points. }
\label{Fig:det}
\end{figure}

For mode II, by definition, we have
\[
\wt{\mc{H}}_{\text{CB}}(\partial_x,\theta)(w_1,w_2)^{\TT}=\dfrac{2(\ka_2+9\ka_4)}
{\veps^2}
\begin{pmatrix}
\tfrac38(3\partial_x^2-\theta^2)&\tfrac34\imath\theta\partial_x\\[.1in] \tfrac34\imath\theta\partial_x&\tfrac38(\partial_x^2-3\theta^2)
\end{pmatrix}\begin{pmatrix}
w_1\\
w_2
\end{pmatrix}=\begin{pmatrix}
0\\
0
\end{pmatrix}.
\]
The nonexistence of mode II eigenfunction is equivalent to the fact
that the the complementing condition for elliptic PDE
\cite{AgmonDouglisNirenberg:1959} is satisfied for the Cauchy-Born
problem. For this linearized elasticity problem with boundary
condition $\sum_{j=1}^3B_{kj}z_j(x)=0$ for $k=1,2,3$, the complementing condition
is valid by~\cite{Thompson:1969} and by explicit calculation
\begin{equation*}
  \ka_2+9\ka_4=h(1)+9h(\sqrt3) \geq 60 K.
\end{equation*}
Hence mode II eigenfunction does not exist.  For mode III, we have
$\zeta\to 1$ as $\theta\to 0$, and $z_1,z_2\to 1$.  The solution takes
the form
\[
\lim_{\theta\to 0}z(x_i,\zeta)=\sum_{k=3}^6c_k(1)\mc{M}_2^{\at}(z_k,1)z_k^i
\]
with $z_3=-0.0042,z_4=0.0293,z_5=0.8945 - 0.0969\I$ and $z_6=\wb{z}_5$. Substituting the above expression into the last two boundary conditions, we obtain
\[\scriptsize
\begin{pmatrix}z_3^{-1}\mc{M}^{\at}_{22}(z_3,1)&z_4^{-1}\mc{M}^{\at}_{22}(z_4,1)&
z_5^{-1}\mc{M}^{\at}_{22}(z_5,1)&z_6^{-1}\mc{M}^{\at}_{22}(z_6,1)\\
z_3^{-1}\mc{M}^{\at}_{12}(z_3,1)&z_4^{-1}\mc{M}^{\at}_{12}(z_4,1)
&z_5^{-1}\mc{M}^{\at}_{12}(z_5,1)&z_6^{-1}\mc{M}^{\at}_{12}(z_6,1)\\
z_3\mc{M}^{\at}_{22}(z_3,1)&z_4\mc{M}^{\at}_{22}(z_4,1)
&z_5\mc{M}^{\at}_{22}(z_5,1)&z_6\mc{M}^{\at}_{22}(z_6,1)\\
z_3\mc{M}^{\at}_{12}(z_3,1)&z_4\mc{M}^{\at}_{12}(z_4,1)
&z_5\mc{M}^{\at}_{12}(z_5,1)&z_6\mc{M}^{\at}_{12}(z_6,1)\end{pmatrix}
\begin{pmatrix}c_3(1)\\c_4(1)\\c_5(1)\\c_6(1)\end{pmatrix}
=\begin{pmatrix}0\\0\\0\\0\end{pmatrix},
\]
which yields $c_i(1)=0$ for $i=3,\cdots,6$. This concludes that there does not exist mode III eigenfunction. Therefore, this scheme is stable.

\section{Conclusion}\label{sec:conclusion}
We have identified stability conditions, especially stability
conditions at the interface, for atomistic-to-continuum hybrid methods
with sharp interface. Under these stability conditions, we establish
convergence of the hybrid scheme. Though we only consider the flat
interface, the analysis can be extended to smooth interface
between the atomistic and continuum regions. In that case, we need to
check the interface stability condition
Assumption~\ref{assump:complement} for interface with different angles.

For the example of atomistic-to-continuum coupling method for a
triangular lattice considered here, the stability can be checked by
hand or with some help of numerical computation. For many other more
complicated methods, this might not be easily done. One possible
direction is to use numerical methods and symbolic computations to check stability
conditions, in analogy to checking GKS conditions for example as
in~\cites{Thune:1986}. This is an interesting future research
direction.

The result in this paper does not apply to transition interface
between atomistic and continuum regions that involves
corners. The coarsening in the continuum region is also not taken into account.
Extension of the results to hybrid schemes with transition
interface with corners and with coarsening in the continuum region would be very interesting.
\bibliographystyle{amsxport}
\bibliography{quasicont}
\end{document}


\maketitle

\section{Lattice function and norms}\label{sec:lattice}

We will consider only Bravais lattices in this work, which is denoted
as $\LL$.  Let $\{a_j\}_{j=1}^d \subset \RR^d$ be the basis vectors of
$\LL$, and $d$ be the dimension,
\[
  \LL = \Bigl\{ x \in \RR^d \,\Big\vert\,x = \sum_j n_j a_j,\, n \in \mathbb{Z}^d \Bigr\}.
\]
Let $\{b_j\}_{j=1}^d \subset \RR^d$ be the reciprocal basis vectors
satisfying \( a_j \cdot b_k = 2\pi \delta_{jk}, \)where $\delta_{jk}$
is the standard Kronecker delta symbol.  The reciprocal lattice
$\LL^{\ast}$ is
\[
\LL^{\ast} = \Bigl\{ x \in \RR^d \,\Big\vert\, x = \sum_j n_j b_j, \, n \in
\mathbb{Z}^d \Bigr\}.
\]
We take a computational domain
\[
\Omega = \Bigl\{ \sum_{j} x_j a_j \ \Big\vert\ x \in [0, 1)^d \Bigr\},
\]
and let $\Omega_{\veps}$ be a grid mesh in $\Omega$ with mesh size
$\veps = 1 / (2N),\, N \in \ZZ_+$:
\[
\Omega_{\veps} = \Bigl\{ x_{\nu} = \veps \sum_{j}  \nu_j a_j\ \Big\vert\
\nu \in \ZZ^d,\, 0 \leq \nu_j < 2N \Bigr\}.
\]
Using the reciprocal basis $\{b_j\}$, we define
\[
\LL_{\veps}^{\ast} = \Bigl\{ \xi = \sum_{j} k_j b_j \ \Big\vert\
k \in \ZZ^d,\, -N \leq k_j < N \Bigr\}.
\]
We will identify functions defined on $\Omega_{\veps}$ with their
periodic extensions in this work, i.e., we consider the
periodic boundary condition. General boundary conditions will be left
for future work.

For $\mu \in \ZZ^d$, we define the translation operator $T^{\mu}_{\veps}$
as
\[
(T^{\mu}_{\veps} u)(x) = u(x + \veps \mu_j a_j)\quad \text{for } x \in
\RR^d,
\]
where the index summation convention is used. We define the forward
and backward difference operators as
\[
\ffd{\veps, \mu} =\veps^{-1}(T^{\mu}_{\veps} - I)\qquad\text{and}\qquad
\bfd{\veps, \mu}=\veps^{-1}(I - T^{-\mu}_{\veps}),
\]
where $I$ denotes the identity operator. We say $\alpha$ is a
multi-index, if $\alpha \in \ZZ^d$ and $\alpha \geq 0$. We will use
the notation $\abs{\alpha} = \sum_{j=1}^d \alpha_j$.
For a multi-index $\alpha$, the difference operator
$D_{\veps}^{\alpha}$ is given by
\[
D^{\alpha}_{\veps} = \prod_{j=1}^d (\ffd{\veps, e_j})^{\alpha_j},
\]
where $\{e_j\}_{j=1}^d$ are the canonical basis of $\RR^d$ (columns of
a $d\times d$ identity matrix).

We will use various norms for functions defined on
$\Omega_{\veps}$. For integer $k\geq 0$, define the difference norm
\[
  \norm{u}_{\veps, k}^2 = \sum_{0 \leq \abs{\alpha} \leq k} \veps^d
  \sum_{x \in \Omega_{\veps}} \abs{(D_{\veps}^{\alpha} u)(x)}^2.
\]
It is clear that $\norm{\cdot}_{\veps, k}$ is a discrete analog of
Sobolev norm associated with $H^k(\Omega)$. Hence, we denote the
corresponding spaces of lattice functions as $H_{\veps}^k(\Omega)$ and
$L_{\veps}^2(\Omega)$ when $k = 0$.  We also need the uniform norms on
$\Omega_{\veps}$, which are defined as
\begin{align*}
  & \norm{u}_{L^{\infty}_{\veps}} = \max_{x \in \Omega_{\veps}}
  \abs{u(x)}, \\
  & \norm{u}_{W^{k,\infty}_{\veps}} = \sum_{0 \leq \abs{\alpha} \leq
    k} \max_{x \in \Omega_{\veps}} \abs{(D^{\alpha}_{\veps} u)(x)}.
\end{align*}
Recall that we identify lattice function $u$ with its periodic
extension to function defined on $\veps \LL$, and hence differences of
the lattice functions are well-defined. These norms may be extended
to vector-valued functions as usual. For $k>d/2$, we have the discrete
Sobolev inequality
$\norm{u}_{L^{\infty}_{\veps}}\lesssim\norm{u}_{\veps,k}$. Here and
throughout this paper, we denote $A \lesssim B$ if $A \leq C B$ with
$C$ an absolute constant.

The discrete Fourier transform for a lattice function $u$ is given for
$\xi \in \LL_{\veps}^{\ast}$ by
\begin{equation*}
  \wh{u}(\xi) = \Bigl(\frac{\veps}{2\pi}\Bigr)^d
  \sum_{x \in \Omega_{\veps}} e^{-\I \xi \cdot x} u(x).
\end{equation*}
By the Fourier inversion formula, for $x \in \Omega_{\veps}$,
\begin{equation*}
  u(x) = \sum_{\xi \in \mathbb{L}^{\ast}_{\veps}} e^{\I x \cdot \xi}
  \wh{u}(\xi).
\end{equation*}

We will use a symbol which plays the same role for the difference
operators as $\Lambda^2(\xi) = 1 + \Lambda_0^2(\xi) = 1 + \abs{\xi}^2$
for the differential operators. For $\veps > 0$, let
\[
  \Lambda_{j, \veps}(\xi) = \frac{1}{\veps} \abs{ e^{\I\veps
      \xi_j} - 1}, \qquad j = 1, \cdots, d,
\]
and
\[
\Lambda_{\veps}^2(\xi) = 1 + \Lambda_{0, \veps}^2(\xi) = 1 +
\sum_{j=1}^d \Lambda_{j, \veps}^2(\xi) = 1 + \sum_{j=1}^d
\frac{4}{\veps^2} \sin^2\Bigl( \frac{\veps \xi_j}{2} \Bigr).
\]
It is not hard to check for any $\xi \in \LL_{\veps}^{\ast}$, there holds
\begin{equation*}
  c\Lambda^2(\xi)
  \leq \Lambda_{\veps}^2(\xi) \leq \Lambda^2(\xi),
\end{equation*}
where the positive constant $c$ depends on $\{b_j\}$.

\section{Proof of Lemma~\ref{lem:Hcons}}
\label{sec:cons}

Let us first recall the following consistency lemma proved in
[22, Section 2] (proofs of these results do not depend
on the smoothness of $\varrho$).
\begin{lemma}[Consistency]\label{lem:cons}
  For any $u$ smooth, we have
  \begin{align}\label{cons:eq1}
    & \norm{\mc{F}_{\at}[u] - \mc{F}_{\CB}[u]}_{L^{\infty}_{\veps}}
    \leq
    C \veps^2 \norm{u}_{W^{18, \infty}}, \\
    & \norm{\mc{F}_{\veps}[u] - \mc{F}_{\CB}[u]}_{L^{\infty}_{\veps}}
    \leq C\veps^2 \norm{u}_{W^{18, \infty}},\\
    & \label{eq:compareqcat} \norm{\mc{F}_{\qc}[u] -
      \mc{F}_{\at}[u]}_{L^{\infty}_{\veps}} \leq C \veps^2
    \norm{u}_{W^{18, \infty}},
  \end{align}
  where the constant $C$ depends on $V$ and $\norm{u}_{L^\infty}$, but
  is independent of $\veps$.
\end{lemma}

\begin{proof}[Proof of Lemma~\ref{lem:Hcons}] The proof for
  \eqref{eq:HconsatCB} and \eqref{eq:Hconsathy} are analogous, and
  hence we will only prove the latter. By definition, for $1 \leq j, k
  \leq d$,
  \begin{align*}
    & ({h}_{\at})_{jk}(\xi) = e^{-\I x \cdot \xi}
    (\mc{H}_{\at} (e_k f_{\xi}))_j(x), \\
    & ({h}_{\qc})_{jk}(x, \xi) = e^{-\I x \cdot \xi}
    (\mc{H}_{\qc} (e_k f_{\xi}))_j(x),
  \end{align*}
  where $f_{\xi}(x) = e^{\I x\cdot \xi}$ for $x \in \Omega$. Taking
  difference of the above two equations, we obtain the bound
  \begin{equation*}
    \abs{{h}_{\at}(\xi) - h_{\qc}(x, \xi) } \leq C
    \sup_{1\leq k \leq d} \norm{\mc{H}_{\at} (e_k f_{\xi}) - \mc{H}_{\qc}
      (e_k f_{\xi})}_{L^{\infty}_{\veps}}.
  \end{equation*}
  Note that by the definition of linearized operators $\mc{H}_{\at}$
  and $\mc{H}_{\qc}$, we have
  \begin{equation*}
    \mc{H}_{\at} (e_k f_{\xi}) - \mc{H}_{\qc} (e_kf_{\xi})
    = \lim_{t \to 0^+} \frac{1}{t} \bigl( \mc{F}_{\at}[t (e_k f_{\xi})] -
    \mc{F}_{\qc}[t (e_k f_{\xi})] \bigr).
  \end{equation*}
  Hence,
  \begin{equation*}
    \begin{aligned}
      \norm{\mc{H}_{\at} (e_kf_{\xi}) - \mc{H}_{\qc}
        (e_kf_{\xi})}_{L^{\infty}_{\veps}} & = \lim_{t \to 0^+}
      \frac{1}{t} \norm{\mc{F}_{\at}[t (e_kf_{\xi})] -
        \mc{F}_{\qc}[t (e_kf_{\xi})]}_{L^{\infty}_{\veps}} \\
      & \lesssim \veps^2 \norm{e_k f_{\xi}}_{W^{18, \infty}} \lesssim
      \veps^2 \norm{e_k f_{\xi}}_{H^{s}} \lesssim \veps^2 (1 +
      \abs{\xi}^2)^{s/2},
    \end{aligned}
  \end{equation*}
  where $s$ is chosen so that the Sobolev inequality $\norm{f}_{W^{18,
      \infty}(\Omega)} \leq C \norm{f}_{H^{s}(\Omega)}$ holds for any
  $f \in H^{s}(\Omega)$ ($s$ depends on the dimension). Here, we have
  used Lemma~\ref{lem:cons} in the first inequality, noticing
  that $\norm{t e_k f_{\xi}}_{L^{\infty}}$ is uniformly bounded for
  $\xi$ as $t \to 0$. This concludes the proof.
\end{proof}

\section{Additional details for Example 1} 
\label{sec:prooflemroot}

\begin{lemma}\label{lem:root}
$z_1, z_2$ and $z_3$ are distinct roots.
\end{lemma}

\begin{proof}
It is clear that
\[
z_2=w_2\zeta^{1/2}\quad\text{and\quad} z_3=w_3\zeta^{1/2}
\]
with $-1<w_3<0<w_2<1$, this implies $z_2\not=z_3$.

A direct calculation gives
\begin{align*}
  z_1&=\dfrac{6-\zeta-\wb{\zeta}-\sqrt{\Lr{4-\zeta-\wb{\zeta}}\Lr{2-\zeta-\wb{\zeta}\,}}}
  {2(1+\wb{\zeta})}\\
  &=\dfrac{6-\zeta-\wb{\zeta}-\sqrt{\Lr{4-\zeta-\wb{\zeta}}\Lr{2-\zeta-\wb{\zeta}\,}}}
  {2(1+\zeta)\Lr{1+\wb{\zeta}}}(1+\zeta)\\
  &=\dfrac{6-\zeta-\wb{\zeta}-\sqrt{\Lr{4-\zeta-\wb{\zeta}}\Lr{2-\zeta-\wb{\zeta}\,}}}
  {2\Lr{2+\zeta+\wb{\zeta}}}\Lr{\zeta^{1/2}+\wb{\zeta}^{1/2}}\zeta^{1/2}.
\end{align*}
Recalling $\zeta=e^{\I\theta}$ with $\theta\in(-\pi,\pi)$, and we may write
\[
z_1=\dfrac{2\cos(\theta/2)}{3-\cos\theta+\sqrt{(7-\cos\theta)(1-\cos\theta)}}\,
e^{\I\theta/2}.
\]
Note that
\[
\dfrac{2\cos(\theta/2)}{3-\cos\theta+\sqrt{(7-\cos\theta)(1-\cos\theta)}}>0>w_3,
\]
this implies $z_1\not=z_3$.

It remains to prove $z_1\not=z_2$. Note that
\[
z_2=\dfrac12\Lr{B-\sqrt{B^2-4}}e^{\I\theta/2}
\]
with
\[
B=-A/2+\sqrt{A^2/4+14-\Lr{\zeta+\wb{\zeta}\,}}.
\]
Using
\[
A=\zeta+\wb{\zeta}+\zeta^3+\wb{\zeta}^3=\Lr{\zeta+\wb{\zeta}\,}
\Lr{\zeta^2+\wb{\zeta}^2}=4\cos(\theta/2)\cos\theta,
\]
we write
\begin{align*}
A^2/4+14-\Lr{\zeta+\wb{\zeta}}&=16\cos^2(\theta/2)\cos^2\theta+14-2\cos\theta\\
&=16\cos^2(\theta/2)\cos^2\theta+14-2(2\cos^2(\theta/2)-1)\\
&=16-4\cos^2(\theta/2)\sin^2\theta.
\end{align*}
This gives
\[
B=2\sqrt{4-\cos^2(\theta/2)\sin^2\theta}-2\cos(\theta/2)\cos\theta.
\]
To prove $z_1\not=z_2$, it remains to show $\abs{z_1}\not=\abs{z_2}$, i.e.,
\[
\dfrac12\Lr{B-\sqrt{B^2-4}}
\not=\dfrac{2\cos(\theta/2)}{3-\cos\theta+\sqrt{(7-\cos\theta)(1-\cos\theta)}}.
\]
Actually, we shall prove that for $\theta\in (-\pi,\pi)$ and $\theta\not=0$, there holds
\begin{equation}\label{eq:compare}
\dfrac12\Lr{B-\sqrt{B^2-4}}
>\dfrac{2\cos(\theta/2)}{3-\cos\theta+\sqrt{(7-\cos\theta)(1-\cos\theta)}}.
\end{equation}
The above inequality is equivalent to
\begin{equation}\label{eq:compare1}
3-\cos\theta+\sqrt{(7-\cos\theta)(1-\cos\theta)}
>\cos(\theta/2)\Lr{B+\sqrt{B^2-4}\,}.
\end{equation}
Denote by $t=\cos(\theta/2)$, we write the above inequality as
\begin{equation}\label{eq:compare2}
2-t^2+\sqrt{(4-t^2)(1-t^2)}
>t\Lr{g(t)+\sqrt{g^2(t)-1}},\qquad t\in [0,1),
\end{equation}
where
\[
g(t){:}=t-2t^3+2\sqrt{1-t^4+t^6}.
\]

To prove~\eqref{eq:compare2}, we firstly
prove
\begin{equation}\label{eq:comparebasic}
2-t^2> tg(t)\qquad t\in [0,1).
\end{equation}
A direct calculation gives
\begin{align*}
2-t^2-tg(t)&=2(1-t^2)+2t\Lr{t^3-\sqrt{1-t^4+t^6}}\\
&=2(1-t^2)+\dfrac{2t(t^4-1)}{\sqrt{1-t^4+t^6}+t^3}\\
&=2(1-t^2)\Lr{1-\dfrac{t+t^3}{\sqrt{1-t^4+t^6}+t^3}}.
\end{align*}
Note that
\[
\sqrt{1-t^4+t^6}> t,
\]
which follows from \((1-t^2)(1-t^4)> 0\). Combining the above two inequalities,
we obtain~\eqref{eq:comparebasic}.

Next, by~\eqref{eq:comparebasic} and note $g(t)\ge 0$, we obtain
\[
(4-t^2)(1-t^2)=(2-t^2)^2-t^2\ge t^2(g^2(t)-1).
\]
A direct calculation gives that \(g(t)\ge 1\). Therefore,
\[
\sqrt{(4-t^2)(1-t^2)}\ge t\sqrt{g^2(t)-1},
\]
which together with~\eqref{eq:comparebasic} gives~\eqref{eq:compare2}. This
implies $z_1\not= z_2$ and completes the proof.
\end{proof}